\documentclass[a4paper,12pt]{article}

\usepackage[utf8]{inputenc}
\usepackage[T1]{fontenc}
\usepackage{lmodern}
\usepackage{graphicx} 
\usepackage{amsmath} 
\usepackage{amssymb}  
\usepackage{color}

\usepackage{amsthm}
\usepackage{hyperref}
\usepackage{subfigure} 
\usepackage{caption}
\usepackage{enumitem}

\def\sizefig{0.40}
\newcommand{\R}{\mathbb{R}}
\newcommand{\Pcal}{\mathcal{P}}

\newcommand{\N}{\mathbb{N}}
\newcommand{\x}{\mathbf{x}}
\newcommand{\y}{\mathbf{y}}
\newcommand{\z}{\mathbf{z}}
\newcommand{\f}{\mathbf{f}} 
\newcommand{\Plam}{\P_{\lambda}}

\def\P{\mathbf{P}}
\def\Q{\mathbf{Q}}
\def\q{\mathbf{q}}
\newcommand{\M}{\mathbf{M}}
\def\m{\mathbf{m}}
\def\H{\mathbf{H}}
\def\h{\mathbf{h}}
\def\a{\mathbf{a}}

\def\S{\mathbf{S}}
\def\K{\mathbf{K}}
\newcommand{\A}{\mathbf{A}}

\newcommand{\Mcal}{\mathcal{M}}
\newcommand{\xlamstar}{{\x}^*({\lambda})}
\theoremstyle{plain}
\newtheorem{theorem}{Theorem}[section]
\newtheorem{lemma}[theorem]{Lemma}

\theoremstyle{definition}
\newtheorem{definition}[theorem]{Definition}
\newtheorem{hypothesis}[theorem]{Assumption}

\newtheorem{example}{Example}

\theoremstyle{remark}
\newtheorem*{remark}{Remark}

\title{\bf Approximating Pareto Curves using Semidefinite Relaxations}

\begin{document}

\author{Victor Magron$^{1,2}$  \and Didier Henrion$^{1,2,3}$ \and Jean-Bernard Lasserre$^{1,2}$}

\footnotetext[1]{CNRS; LAAS; 7 avenue du colonel Roche, F-31400 Toulouse; France.}
\footnotetext[2]{Universit\'e de Toulouse;  LAAS, F-31400 Toulouse, France.}
\footnotetext[3]{Faculty of Electrical Engineering, Czech Technical University in Prague,
Technick\'a 2, CZ-16626 Prague, Czech Republic}

\maketitle
\begin{abstract}
We consider the problem of constructing an approximation of the Pareto curve associated with the multiobjective optimization problem $\min_{\mathbf{x} \in \mathbf{S}}\{ (f_1(\mathbf{x}), f_2(\mathbf{x}))  \}$, where $f_1$ and $f_2$ are two conflicting polynomial criteria and $\mathbf{S} \subset \mathbb{R}^n$ is a compact basic semialgebraic set. 
We provide a systematic numerical scheme to approximate the Pareto curve. We start by reducing the initial problem into a scalarized polynomial optimization problem (POP). Three scalarization methods lead to consider different parametric POPs, namely (a) a weighted convex sum approximation, (b) a weighted Chebyshev approximation, and (c) a parametric sublevel set approximation.
For each case, we have to solve a semidefinite programming (SDP) hierarchy parametrized by the number of moments or equivalently
the degree of a polynomial sums of squares approximation of the Pareto curve.  When the degree
of the polynomial approximation tends to infinity, we provide guarantees of convergence to the Pareto curve in $L^2$-norm
for methods (a) and (b), and $L^1$-norm for method (c).
\end{abstract}
\section{Introduction}
\label{sec:intro}
Let $\P$ be the bicriteria polynomial optimization problem $\min_{\x \in \S}\{ (f_1(\x), f_2(\x))  \}$, where $\S \subset \R^n$ is the basic semialgebraic set:
\begin{align}
\label{eq:S}
\S := \{\x \in \R^n : g_1(\x)  \geq 0, \dots, g_m(\x)  \geq 0 \} \:,
\end{align}

for some polynomials $f_1, f_2, g_1, \dots, g_m \in \R[\x]$. Here, we assume the following:
\begin{hypothesis}
\label{hyp:order}
The image space $\R^2$ is partially ordered with  the positive orthant $\R_{+}^2$. That is,
given $\x \in \R^2$ and $\y \in \R^2$, it holds $\x \geq \y$ whenever $\x-\y \in \R_+^2$. 
\end{hypothesis}
For the multiobjective optimization problem~$\P$, one is usually interested in computing, or at least approximating, the
following optimality set, defined e.g. in \cite[Definition 11.5]{jahn:2010:vector}.
\begin{definition}
Let Assumption~\ref{hyp:order} be satisfied. A point $\bar{\x} \in \S$ is called an {\it Edgeworth-Pareto (EP) optimal point} of Problem~$\P$, when there is no $\x \in \S$ such that $f_j(\x) \leq f_j (\bar{\x}), \: j = 1,2$ and $f(\x) \neq f(\bar{\x})$. A point $\bar{\x} \in \S$ is called a  {\it weakly Edgeworth-Pareto optimal point} of Problem~$\P$, when there is no $\x \in \S$ such that $f_j(\x) < f_j (\bar{\x}), \: j = 1,2$.
\end{definition}
In this paper, for conciseness, we will also use the following terminology:
\begin{definition}
The image set of weakly Edgeworth-Pareto optimal points is called the {\it Pareto curve}.
\end{definition}
Given a positive integer $p$ and $\lambda\in [0,1]$ both fixed, a common workaround consists in solving the {\it scalarized} problem:
\begin{align}
\label{eq:flamx-0}
f^p(\lambda) := \min_{\x \in \S} \{ [(\lambda f_1(\x))^p + ((1 - \lambda) f_2(\x))^p]^{1/p} \}\:,
\end{align}
which includes the weighted sum approximation ($p = 1$)
\begin{align}
\label{eq:flamx}
\Plam^{1} : \: f^1(\lambda) := \min_{\x \in \S} (\lambda f_1(\x) + (1 - \lambda)f_2(\x)) \: ,
\end{align}
and the weighted Chebyshev approximation ($p = \infty$)
\begin{align}
\label{eq:flamx-infty}
\Plam^{\infty} : \: f^{\infty}(\lambda) := \min_{\x \in \S} \max\{\lambda f_1(\x), (1 - \lambda) f_2(\x)\} \: .
\end{align}
Here, we assume that for almost all (a.a.) $\lambda \in [0, 1]$, the solution $\xlamstar$ of the scalarized problem~\eqref{eq:flamx} (resp.~\eqref{eq:flamx-infty}) is unique. 
Non-uniqueness may be tolerated on a Borel set $B \subset [0, 1]$, in which case one assumes image uniqueness of the solution. 
Then, by computing a solution $\xlamstar$, one can approximate the set 
$\{(f_1^*(\lambda), f_2^*(\lambda)) : \lambda \in [0, 1]\}$, where
\[
f_j^*(\lambda) := f_j(\xlamstar), \: j=1,2.
\]
%
Other approaches include using a numerical scheme such as the modified Polak method \cite{Polak:1976:Multicriteria}: first, one considers a finite discretization $(y_1^{(k)})$ of the interval $[a_1, b_1]$, where 
\begin{equation}
\label{eq:interval1}
a_1 := \min_{\x  \in \S} f_1(\x), \quad b_1 := f_1(\overline{\x}) \:, 
\end{equation}
with $\overline{\x}$ being a solution of $\min_{\x \in \S} f_2(\x)$. Then, for each $k$, one computes an optimal solution $\x_k$ of the constrained optimization problem $y^{(k)}_2:=\min_{\x \in \S} \{f_2 (\x) : f_1(\x) = y_1^{(k)} \}$ and select the Pareto front from
the finite collection $\{(y_1^{(k)},y^{(k)}_2)\}$.
This method can be improved with the iterative Eichfelder-Polak algorithm, see e.g.~\cite{Eichfelder:2009:Scalarization}. Assuming the smoothness of the Pareto curve, one can use the Lagrange multiplier of the equality constraint to select the next point  $y_1^{(k + 1)}$. It allows to combine the adaptive control of discretization points with the modified Polak method. In~\cite{DasDennis:1998:NBI}, Das and Dennis introduce the Normal-boundary intersection method which can find a uniform spread of points on the Pareto curve with more than two conflicting criteria and without assuming that the Pareto curve
is either connected or smooth. However, there is no guarantee that the NBI method succeeds in general and even in case it works well, the spread of points is only uniform under certain additional assumptions.
Interactive methods such as STEM~\cite{BenayounMontgolfier:1971:MultiLinear} rely on a {\it decision maker} to select at each iteration the weight $\lambda$ (most often in the case $p = \infty$) and to make a trade-off between criteria after solving the resulting scalar optimization problem.

So discretization methods suffer from two major drawbacks. (i) They only provide a {\it finite subset} of the Pareto curve
and (ii) for each discretization point one has to compute a {\it global} minimizer of the resulting optimization problem
(e.g. (\ref{eq:flamx}) or (\ref{eq:flamx-infty})).
Notice that when $f$ and $\S$ are both convex then point (ii) is not an issue. 

In a recent work \cite{Gorissen2012319}, Gorissen and den Hertog  avoid discretization schemes for 
convex problems with multiple linear criteria $f_1, f_2, \dots, f_k$ and a convex polytope $\S$. They provide an inner approximation of $f(\S) + \R_+^k$ by combining robust optimization techniques with semidefinite programming; for more details the reader is referred to~\cite{Gorissen2012319}.

\paragraph*{Contribution.} We provide a numerical scheme with two characteristic features:
It  avoids a discretization scheme and approximates the Pareto curve in a relatively strong sense. 
More precisely, the idea is 
consider multiobjective optimization as 
a particular instance of {\it parametric polynomial optimization}
for which some strong approximation results are available when the data are polynomials and semi-algebraic sets. In fact we will investigate this approach: 
\begin{description}
\item[method (a)] for the first formulation (\ref{eq:flamx}) when $p=1$, this is a {\it weighted convex sum approximation};
\item[method (b)] for the second formuation (\ref{eq:flamx-infty}) when $p=\infty$, this is a {\it weighted Chebyshev approximation};
\item[method (c)] for a third formulation inspired by~\cite{Gorissen2012319}, this is a {\it parametric sublevel set approximation}.
\end{description}

When using some weighted combination of criteria ($p=1$, method (a) or $p=\infty$, method (b))
we treat each function $\lambda\mapsto f_j(\lambda)$, $j=1, 2$, as the signed density of the signed Borel measure $d \mu_j := f_j(\lambda) d \lambda$ with respect to the Lebesgue measure $d \lambda$ on $[0, 1]$. Then the procedure consists of two distinct steps:
\begin{enumerate}
\item
In a first step, we solve a hierarchy of semidefinite programs (called SDP hierarchy) which permits to approximate any finite number $s + 1$ of moments $\m_j := (m_j^k), k = 0, \dots, s$ where :
\[ m_j^k := \int_0^1 \lambda^k f_j^*(\lambda) d \lambda \:, \quad k = 0, \dots, s \:, j=1,2 \:.  \]
More precisely, for any fixed integer $s$, step $d$ of the SDP hierarchy provides an approximation $\m^d_j$ of $\m_j$ which converges to
$\m_j$ as $d\to\infty$.
\item
The second step consists of two {\it density estimation} problems: namely, for each $j=1,2$, and given the moments $\m_j$ of the measure $f_j^* d \lambda$ with unknown density $f_j^*$ on $[0, 1]$, one computes a univariate polynomial $h_{s, j} \in \R_{s}[\lambda]$ which solves the optimization problem $\min_{h \in \R_{s}[\lambda]} \int_0^1 (f_j^*(\lambda) - h)^2 d \lambda$
if the moments $\m_j$ are known exactly.
The corresponding vector of coefficients $\h_j^s \in \R^{s + 1}$ is given by $\h_j^s = \H_s(\lambda)^{-1} \m_j$, $j=1,2$, where $\H_s(\lambda)$ is the $s$-moment matrix of the Lebesgue measure $d \lambda$ on $[0, 1]$; therefore in the expression 
for $\h_j^s$ we replace
$\m_j$ with its approximation. 
\end{enumerate}
 Hence for both methods (a) and (b), we have {\it $L^2$-norm convergence guarantees}. 
 
Alternatively, in our method (c), one can estimate the Pareto curve by solving for each $\lambda \in [a_1, b_1]$ the following parametric POP:
\begin{equation}\label{eq:flamx-u}
\Plam^u: \quad f^u(\lambda) := \min_{\x \in \S} \,\{\, f_2 (\x) : f_1 (\x) \leq \lambda \,\} \enspace,
\end{equation}
with $a_1$ and $b_1$ as in~\eqref{eq:interval1}. Notice that by definition $f^u(\lambda)=f^*_2(\lambda)$.
Then, we derive an SDP hierarchy parametrized by $d$, so that the optimal solution $q_{2d} \in \R[\lambda]_{2 d}$ of the $d$-th relaxation underestimates $f_2^*$ over $[a_1, b_1]$. In addition, $q_{2 d}$ converges to $f_2^*$ with respect to the $L_1$-norm, as $d \to \infty$. In this way, one can approximate from below the set of Pareto points, as closely as desired. 
 Hence for method (c), we have {\it $L^1$-norm convergence guarantees}. 

It is important to observe that even though $\Plam^{1}$, $\Plam^{\infty}$ and $\Plam^u$ are all global optimization problems
we do {\it not} need to solve them exactly. In all cases the information provided at step $d$ of the SDP hierarchy
(i.e. $\m^d_j$ for $\Plam^{1}$ and $\Plam^{\infty}$ and the polynomial
$q_{2 d}$ for $\Plam^u$) permits to define an approximation  of the Pareto front. In other words even in the absence of convexity
the SDP hierarchy allows to approximate the Pareto front and of course the higher in the hierarchy the better is the approximation.

The paper is organized as follows. Section~\ref{sec:prelim} is dedicated to recalling
some background about moment and localizing matrices. Section~\ref{sec:approx}
describes our framework to approximate the set of Pareto points using SDP relaxations of parametric optimization programs. These programs are presented in Section~\ref{subsec:parampop} while we describe how to reconstruct the Pareto curve in Section~\ref{subsec:sdp}. Section~\ref{sec:bench} presents some numerical experiments which illustrate the different approximation schemes.
\if{
The weighted sum approach is treated in Section~\~\eqref{subsubsec:cvx}, the Chebyshev norm approximation in Section~\ref{subsubsec:cheb}, while Section~\ref{subsubsec:L1} describes an a.
These programs include the weighted sum approach (Section), the Chebyshev norm () and 
}\fi
\section{Preliminaries}
\label{sec:prelim}
Let $\R[\lambda,\x]$ (resp.~$\R[\lambda,\x]_{2d}$) denote the ring of real polynomials (resp. of degree at most $2d$) in the variables
$\lambda$ and $\x=(x_1,\ldots,x_n)$, whereas $\Sigma[\lambda,\x]$ (resp.~$\Sigma[\lambda,\x]_d$) denotes 
its subset of sums of squares (SOS) of polynomials (resp.~of degree at most $2d$).
For every $\alpha\in\N^n$ the notation $\x^\alpha$ stands for the monomial $x_1^{\alpha_1}\dots x_n^{\alpha_n}$ and for every $d\in\N$, let $\N^{n+1}_d := \{ \beta \in \N^{n+1} : \sum_{j=1}^{n+1} \beta_j \leq d \}$, whose cardinal is $s_n(d)= \binom{n+1+d}{d}$.
A polynomial $f\in\R[\lambda,\x]$ is written 
\[(\lambda,\x)\mapsto f(\lambda,\x)\,=\,\sum_{(k,\alpha)\in\N^{n+1}}\,f_{k\alpha}\,\lambda^k\x^\alpha \:, \]
and $f$ can be identified with its vector of coefficients $\f=(f_{k\alpha})$ in the canonical basis $(\x^\alpha)$, $\alpha\in\N^n$.
For any symmetric matrix $\A$ the notation $\A\succeq0$ stands for $\A$ being semidefinite positive.
A real sequence $\z=(z_{k\alpha})$, $(k,\alpha)\in\N^{n+1}$, has a {\it representing measure} if
there exists some finite Borel measure $\mu$ on $\R^{n+1}$ such that 
\[z_{k\alpha}\,=\,\int_{\R^{n+1}}\lambda^k\x^\alpha\,d\mu(\lambda,\x), \quad \forall (k,\alpha) \in \N^{n+1} \: .\]
Given a real sequence $\z=(z_{k\alpha})$ define the linear functional $L_\z:\R[\lambda,\x]\to\R$ by:
\[f\:(=\sum_{(k,\alpha)} f_{k\alpha}\lambda^k\x^\alpha)\quad\mapsto L_\z(f)\,=\,\sum_{(k,\alpha)}f_{k\alpha}\,z_{k\alpha},\quad f\in\R[\lambda,\x] \:.\]
\paragraph*{Moment matrix}
The {\it moment} matrix associated with a sequence
$\z=(z_{k\alpha})$, $(k,\alpha)\in\N^{n+1}$, is the real symmetric matrix $\M_d(\z)$ with rows and columns indexed by $\N^{n+1}_d$, and whose entry $(i,\alpha),(j,\beta)$ is just $z_{(i+j)(\alpha+\beta)}$, for every $(i,\alpha),(j,\beta)\in\N^{n+1}_d$. 

If $\z$ has a representing measure $\mu$ then
$\M_d(\z)\succeq0$ because
\[\langle\f,\M_d(\z)\f\rangle\,=\,\int f^2\,d\mu\,\geq0,\quad\forall \,\f\,\in\R^{s_n(d)}.\]

\paragraph*{Localizing matrix}
With $\z$ as above and $g\in\R[\lambda,\x]$ (with $g(\lambda,\x)=\sum_{\ell,\gamma} g_{\ell\gamma}\lambda^\ell\x^\gamma$), the {\it localizing} matrix associated with $\z$ 
and $g$ is the real symmetric matrix $\M_d(g\,\z)$ with rows and columns indexed by $\N^n_d$, and whose entry $((i,\alpha),(j,\beta))$ is just $\sum_{\ell,\gamma}g_{\ell\gamma} z_{(i+j+\ell)(\alpha+\beta+\gamma)}$, for every 
$(i,\alpha),(j,\beta)\in\N^{n+1}_d$.

If $\z$ has a representing measure $\mu$ whose support is 
contained in the set $\{\x\,:\,g(\x)\geq0\}$ then
$\M_d(g\,\z)\succeq0$ because
\[\langle\f,\M_d(g\,\z)\f\rangle\,=\,\int f^2\,g\,d\mu\,\geq0,\quad\forall \,\f\,\in\R^{s_n(d)} \: .\]

In the sequel, we assume that $\S := \{ \x \in \R^n : g_1 (\x) \geq 0, \dots,  g_m (\x) \geq 0 \}$ is contained in a box. It ensures that there is some integer $M > 0$ such that the quadratic polynomial $g_{m+1} (\x) := M - \sum_{i = 1}^n x_i^2$ is nonnegative over $\S$. Then,  we add the redundant polynomial constraint $g_{m + 1} (\x) \geq 0 $ to the definition of $\S$.

\section{Approximating the Pareto Curve}
\label{sec:approx}
\subsection{Reduction to Scalar Parametric POP}
\label{subsec:parampop}
Here, we show that computing the set of Pareto points associated with Problem~$\P$ can be achieved with three different parametric polynomial problems. Recall that the feasible set of Problem~$\P$ is $\S := \{ \x \in \R^n : g_1 (\x) \geq 0, \dots,  g_{m+1} (\x) \geq 0 \}$.
\paragraph*{Method (a): convex sum approximation}
\label{subsubsec:cvx}
Consider the scalar objective function $f(\lambda, \x) := \lambda f_1(\x) + (1 - \lambda) f_2(\x)$, $\lambda \in [0, 1]$. 
Let $\K^{1} := [0, 1] \times \S$. Recall from (\ref{eq:flamx}) that function $f^{1}:[0,1]\to\R$ is the optimal value of
Problem~$\Plam^{1}$, i.e.~$f^{1}(\lambda) = \min_\x \{f(\lambda,\x) \: :\: (\lambda,\x) \in \K^1\}$.
If the set $f(\S) + \R_+^2$ is convex, then one can recover  the Pareto curve 
by computing $f^{1}(\lambda)$, for all $\lambda \in [0, 1]$, see~\cite[Table 11.5]{jahn:2010:vector}.
\begin{lemma}
\label{th:fscvx}
Assume that $f(\S) + \R_+^2$ is convex. Then, a point $\overline{\x} \in \S$ belongs to the set of EP points of Problem~$\P$ if and only if there exists some weight $\lambda \in [0, 1]$ such that $\overline{\x}$ is an image unique optimal solution of Problem~$\Plam^{1}$.
\end{lemma}
\paragraph*{Method (b): weighted Chebyshev approximation }
\label{subsubsec:cheb}
Reformulating Problem~$\P$ using the Chebyshev norm approach is more suitable when the set $f(\S) + \R_+^2$ is not convex. We optimize the scalar criterion $f(\lambda, \x) :=  \max[\lambda f_1(\x), (1 - \lambda) f_2(\x)]$, $\lambda \in [0, 1]$. In this case, we assume without loss of generality that both $f_1$ and $f_2$ are positive. Indeed, for each $j=1, 2$, one can always consider the criterion $\tilde{f_j} := f_j - a_j$, where $a_j$ is any lower bound of the global minimum of $f_j$ over $\S$. Such bounds can be computed efficiently by solving polynomial optimization problems using an SDP hierarchy, see e.g.~\cite{Lasserre:2001:moments}.
In practice, we introduce a lifting variable $\omega$ to represent the $\max$ of the objective function. For scaling purpose, we introduce the constant $C := \max (M_1, M_2)$, with $M_j := \max_{\x \in \S} f_j$, $j=1, 2$. 
Then, one defines the constraint set
$\K^{\infty} := \{ (\lambda, \x, \omega) \in \R^{n + 2} : \x \in \S, \lambda \in [0, 1], \lambda f_1 (\x) / C \leq \omega, (1 - \lambda) f_2 (\x) / C \leq \omega  \}$,
\if{\[ \K^{\infty} := \{ (\lambda, \x, \omega) \in \R^{n + 2} : \x \in \S, \lambda \in [0, 1], \lambda f_1 (\x) / C \leq \omega, (1 - \lambda) f_2 (\x) / C \leq \omega  \}  \: , \]}\fi
 which leads to the reformulation of $\Plam^{\infty} : f^{\infty}(\lambda) = \min_{\x, \omega} \{ \omega : (\lambda, \x, \omega) \in \K^{\infty} \}$ consistent with (\ref{eq:flamx-infty}). 
The following lemma is a consequence of~\cite[Corollary 11.21 (a)]{jahn:2010:vector}.
\begin{lemma}
\label{th:fsnoncvx}
Suppose that $f_1$ and $f_2$ are both positive. Then, a point $\overline{\x} \in \S$  belongs to the set of  EP points of Problem~$\P$ if and only if there exists some weight $\lambda \in (0, 1)$ such that $\overline{\x}$ is an image unique optimal solution of Problem~$\Plam^{\infty}$.
\end{lemma}

\paragraph*{Method (c): parametric sublevel set approximation }
\label{subsubsec:L1}
Here, we use an alternative method inspired by~\cite{Gorissen2012319}. Problem~$\P$ can be approximated using the criterion $f_2$ as the objective function and the constraint set
\[
\K^u := \{(\lambda, \x) \in [0, 1] \times \S : (f_1 (\x) - a_1) / (b_1 - a_1) \leq \lambda\},
\] 
 which leads to the parametric POP $\Plam^u : f^u (\lambda) = \min_{\x} \{f_2 (\x) :  (\lambda, \x) \in \K^u\}$
which is consistent with (\ref{eq:flamx-u}), and such that $f^u(\lambda)=f^*_2(\lambda)$ for all $\lambda \in [0,1]$, 
with $a_1$ and $b_1$ as in~\eqref{eq:interval1}.
\begin{lemma}
\label{th:f2u}
Suppose that  $\overline{\x} \in \S$ is an optimal  solution of Problem~$\Plam^u$, with $\lambda\in[0,1]$. Then $\overline{\x}$
belongs to the set of weakly EP points of Problem~$\P$.
\end{lemma}
\begin{proof}
\label{pr:f2u}
Suppose that there exists $\x \in \S$ such that $f_1(\x) < f_1 (\overline{\x})$ and $f_2(\x) < f_2 (\overline{\x})$. Then $\x$ is feasible for Problem~$\Plam^u$ (since $(f_1(\x) - a_1) / (b_1 - a_1) \leq \lambda$) and $f_2 (\overline{\x}) \leq f_2 (\x)$, which leads to a contradiction.
\end{proof}
Note that if a solution $\x^*(\lambda)$ is unique then it is EP optimal. Moreover, if a solution $\x^*(\lambda)$ of Problem $\P^u(\lambda)$ solves also the optimization problem $\min_{\x \in \S} \{f_1 (\x) : f_2(\x) \leq \lambda\}$, then it is an EP optimal point (see~\cite{miett99} for more details).
\if{
\begin{remark}
\label{rk:f2u}
Solving Problem~$\P_u$, $u \in [a_1, b_1]$ does not allow to recover all the weakly EP optimal points. Indeed, ...
\end{remark}
}\fi

\subsection{A Hierarchy of Semidefinite Relaxations}
\label{subsec:sdp}
Notice that the three problems~$\Plam^{1}$, $\Plam^{\infty}$ and $\Plam^u$ are particular instances of the generic parametric optimization problem $f^*(y) := \min_{(y, \x) \in \K} f(y, \x)$. The feasible set $\K$ (resp.~the objective function $f^*$) corresponds to $\K^{1}$ (resp.~$f^{1}$) for Problem~$\Plam^{1}$, $\K^{\infty}$ (resp.~$f^{\infty}$) for Problem~$\Plam^{\infty}$ and $\K^u$ (resp.~$f^u$) for Problem~$\Plam^u$.
We write $\K := \{ (y, \x) \in \R^{n' + 1} : p_1(y, \x) \geq 0, \dots,  p_{m'}(y, \x) \geq 0 \}$. Note also that $n' = n$ (resp.~$n' = n + 1$) when considering Problem~$\Plam^{1}$ and Problem~$\Plam^u$ (resp.~Problem~$\Plam^{\infty}$).

Let $\Mcal(\K)$ be the space of probability measures supported on $\K$.
The function $f^*$ is well-defined because $f$ is a polynomial and $\K$ is compact.
Let $\a=(a_k)_{k \in \N}$, with $a_k=1/(k+1)$, $\forall k \in \N$ and consider the optimization problem:
\begin{equation}
\label{eq:infsdp}
\Pcal:\quad \left\{
\begin{array}{rll}
\rho := \displaystyle\min_{\mu \in \Mcal (\K)} &\displaystyle\int_{\K} f(y,\x) \, d\mu(y,\x)\\
\mbox{s.t.}&\displaystyle\int_{\K} y^k d\mu(y,\x) = a_k, \: k \in \N . \\
\end{array} \right.
\end{equation}
\begin{lemma}
\label{th:1}
The optimization problem $\Pcal$ has an optimal solution $\mu^*\in \Mcal(\K)$ and
if $\rho$ is as in~\eqref{eq:infsdp} then 
\begin{equation}
\label{eq:lem1-1}
\rho = \int_{\K}  f(y,\x) \,d\mu^*\,=\, \displaystyle\int_0^1 f^* (y) \, d y \: .
\end{equation}
Suppose that for almost all (a.a.) $y \in [0,1]$, the parametric optimization problem 
$f^*(y)=\min_{(y, \x) \in \K} f(y, \x)$ has a unique global minimizer $\x^*(y)$ and let $f_j^* : [0,1] \to \R$ be the function $y \mapsto f_j^*(y):=f_j(\x^*(y))$, $j=1,2$. Then for Problem~$\Plam^{1}$, $\rho  = 
\int_0^1 \lambda f_1^*(\lambda) + (1 - \lambda) f_2^*(\lambda) \, d \lambda$, for Problem~$\Plam^{\infty}$, $\rho = 
\int_0^1 \max \{\lambda f_1^*(\lambda), (1 - \lambda) f_2^*(\lambda) \}\, d \lambda$ and for Problem~$\Plam^u$, $\rho = 
\int_0^1 f_2^*(\lambda) \, d\lambda$.
\if{Then for Problem~$\Plam^{1}$,
\begin{equation}
\label{eq:lem1-2}
\rho  = 
\displaystyle\int_0^1 \lambda f_1^*(\lambda) + (1 - \lambda) f_2^*(\lambda) \, d \lambda \: ,
\end{equation}
for Problem~$\Plam^{\infty}$,
\begin{equation}
\label{eq:lem1-3}
\rho = 
\displaystyle\int_0^1 \max \{\lambda f_1^*(\lambda), (1 - \lambda) f_2^*(\lambda) \}\, d \lambda \: ,
\end{equation}
and for Problem~$\Plam^u$,
\begin{equation}
\label{eq:lem1-4}
\rho = 
\displaystyle\int_0^1 f_2^*(\lambda) \, d\lambda \: .
\end{equation}
}\fi
\end{lemma}
\begin{proof}
The proof of~\eqref{eq:lem1-1} follows from~\cite[Theorem 2.2]{Lasserre:2010:JMA} with $y$ in lieu of $\y$. 
Now, consider the particular case of Problem~$\Plam^{1}$. If $\Plam^{1}$
has a unique optimal solution $\x^*(\lambda) \in \S$ for a.a.~$\lambda \in [0,1]$ then
$f^*(\lambda) = \lambda f_1^*(\lambda) + (1 - \lambda)f_2^*(\lambda)$ for a.a.~$\lambda \in [0,1]$. The proofs for $\Plam^{\infty}$ and $\Plam^{u}$ are similar.  
\end{proof}

We set $p_0 := 1$,  $v_l := \lceil \deg p_l  / 2 \rceil$, $l=0,\dots,m'$ and $d_0 := \max (\lceil d_1  / 2 \rceil, \lceil d_2  / 2\rceil, v_1, \dots, v_{m'})$.
Then, consider the following semidefinite relaxations for $d \geq d_0$:
\begin{equation}
\label{eq:primalsdp}
\left\{\begin{array}{rl}
\min_\z &L_\z(f)\\
\mbox{s.t.}&\M_d(\z)\succeq0 \:,\\
&\M_{d-v_l}(p_l\,\z)\succeq0, \: l=1,\ldots,m' \: ,\\
&L_\z(y^k) =a_k, \quad k=0,\dots, 2 d \: .
\end{array}\right.
\end{equation}
\begin{lemma}
\label{lem2}
Assume that for a.a.~$y \in [0,1]$, the parametric optimization problem $f^*(y)=\min_{(y, \x) \in \K} f(y, \x)$ has a unique global minimizer $\x^*(y)$, and let
$\z^d=(z^d_{k\alpha})$, $(k,\alpha)\in\N^{n+1}_{2d}$, be an optimal solution of~\eqref{eq:primalsdp}.
Then
\begin{equation}
\label{eq:lem2-1}
\lim_{d\to\infty}z^d_{k\alpha}= \int_0^1 y^k\,(\x^*(y))^\alpha d y \: .
\end{equation}
In particular, for $s \in \N$, for all $k=0, \dots,s$, $j=1,2$,
\begin{equation}
\label{eq:lem2-2}
m^k_j:=\lim_{d\to\infty}\sum_{\alpha}f_{j\alpha}z^d_{k\alpha}=\int_0^1 y^k\,f^*_j(y)\,d y \:.
\end{equation}
\end{lemma}
\begin{proof}
Let $\mu^*\in \Mcal(\K)$ be an optimal solution of problem $\Pcal$. From
\cite[Theorem 3.3]{Lasserre:2010:JMA},
\[\lim_{d\to\infty}z^d_{k\alpha} = \int_{\K} y^k\x^\alpha\,d\mu^*(y,\x) = \int_0^1 y^k \,(\x^*(y))^\alpha d y \: ,\]
which is~\eqref{eq:lem2-1}. Next, from (\ref{eq:lem2-1}), one has for $s \in \N$:
\[\lim_{d\to\infty}\sum_{\alpha}f_{j\alpha}z^d_{k\alpha} =  
\int_0^1 y^k\,f_j(\x^*(y))\,d y = 
\int_0^1 y^k\,f^*_j(y)\,d y \:,
\]
\if{
\begin{eqnarray*}
\lim_{d\to\infty}\sum_{\alpha}f_{j\alpha}z^d_{k\alpha}&=&\int_{\K}y^k f_j(\x) \, d\mu^*(y,\x)\\
&=&\int_0^1 y^k\,f_j(\x^*(y))\,d y
\,=\,\int_0^1 y^k\,f^*_j(y)\,d y \:, \quad k=0, \dots, s \:, \quad j=1,2 \:,
\end{eqnarray*}
}\fi
for all $k=0, \dots, s$, $j=1,2$. Thus~\eqref{eq:lem2-2} holds.
\end{proof}
The dual of the SDP~\eqref{eq:primalsdp} reads:
\begin{equation}
\label{eq:dualsdp}
\left\{\begin{array}{rl}
\rho_d^* := & \displaystyle\max_{q, (\sigma_l)}  \displaystyle\int_{0}^1 q (y) \, d y \: (= \displaystyle\sum_{k=0}^{2 d} q_k \, a_k) \\
\mbox{s.t.} & f (y, \x) -q(y)= \sum_{k=0}^{m'} \sigma_l(y, \x) \, p_l (y, \x) \:, \forall y \:, \forall \x  \:,\\
& q \in \R[y]_{2d} , \sigma_l \in \Sigma[y,\x]_{d - v_l}, \: l=0,\ldots,m' \: .\\
\end{array}\right.
\end{equation}
%
\begin{lemma}
\label{th:dual}
Consider the dual semidefinite relaxations defined in~\eqref{eq:dualsdp}. Then, one has:
\begin{enumerate}[noitemsep,topsep=0pt,label={(\roman*)}]
\item $\rho_d \uparrow \rho$ as $d \to \infty$.
\item Let $q_{2 d}$ be a nearly optimal solution of~\eqref{eq:dualsdp}, i.e., such that 
$\int_0^1q_{2d}(y)dy\geq\rho^*_d-1/d$. Then $q_{2d}$ underestimates $f^*$ over $\S$ and $\lim_{d \to \infty} \int_0^1 | f^* (y) - q_{2 d} (y) | d \mu = 0$.
\if{
\[ \lim_{d \to \infty} \int_0^1 | f^* (y) - q_{2 d} (y) | d \mu  \: .\]
}\fi
\end{enumerate}
\end{lemma}
\begin{proof}
\label{pr:dual}
It follows from~\cite[Theorem 3.5]{Lasserre:2010:JMA}.
\end{proof}
Note that one can directly approximate the Pareto curve from below when considering Problem~$\Plam^u$. Indeed, solving the dual SDP~\eqref{eq:dualsdp} yields polynomials that underestimate the function $\lambda \mapsto f_2^*(\lambda)$ over $[0, 1]$.
\begin{remark}
\label{rk:hertog}
In~\cite[Appendix A]{Gorissen2012319}, the authors derive the following relaxation from Problem~$\Plam^u$:
\begin{equation}
\label{eq:dualsdphertog}
\left\{\begin{array}{rl}
 \displaystyle\max_{q \in \R[y]_{d}}  & \displaystyle\int_{0}^1 q (\lambda) \, d\lambda\:,\\
\mbox{s.t.} & f_2 (\x) \geq q(  \frac{f_1(\x) - a_1}{b_1 - a_1} ) \:, \forall \x \in \S \:.\\
\end{array}\right.
\end{equation}
Since one wishes to approximate the Pareto curve,
suppose that in (\ref{eq:dualsdphertog}) 
one also imposes that $q$ is nonincreasing over $[0, 1]$. For even degree approximations, the formulation~\eqref{eq:dualsdphertog} is equivalent to
\begin{equation}
\label{eq:dualsdphertogeq}
\left\{\begin{array}{rl}
\displaystyle\max_{q \in \R[y]_{2 d}}  & \displaystyle\int_{0}^1 q (\lambda) \, d\lambda \:,\\
\mbox{s.t.} & f_2 (\lambda) \geq q(\lambda)  \:, \forall \lambda \in [0, 1] \:,\\ 
& \frac{f_1(\x) - a_1}{b_1 - a_1} \leq \lambda \:, \forall \lambda \in [0, 1]\:, \forall \x \in \S \: .\\
\end{array}\right.
\end{equation}
Thus, our framework is related to~\cite{Gorissen2012319} by observing that \eqref{eq:dualsdp} is a strengthening of~\eqref{eq:dualsdphertogeq}.
\end{remark}
%
%
When using the reformulations $\Plam^{1}$ and $\Plam^{\infty}$, computing the Pareto curve is computing (or at least providing good approximations)
of the functions $f^*_j:[0,1]\to\R$ defined above, and
we consider this problem as an {\it inverse problem from generalized moments}.

$\bullet$ For any fixed $s \in \N$, we first compute approximations $\m^{s d}_j = (m_j^{k d})$, $k = 0, \dots, s$, $d \in \N$, of the generalized moments $m^k_j\,\:=\,\int_0^1\lambda^k\,{f^*_j(\lambda)\,d\lambda}, \: k=0,\dots, s, \: j=1,2$,
\if{
\begin{equation}
\label{eq:mom}
m^k_j\,\:=\,\int_0^1\lambda^k\,\underbrace{f^*_j(\lambda)\,d\lambda}_{d\mu_j(\lambda)} \:, \quad k=0,\dots, s \:, \quad j=1,2 \:,
\end{equation}
}\fi
with the convergence property $(\m^{s d}_j)  \to \m^s _j$ as $d\to\infty$, for each $j=1,2$.

$\bullet$ Then we solve the inverse problem: given a (good) approximation  $(\m^{s d}_j)$ of $ \m^s _j$,
find a polynomial $h_{s, j}$ of degree at most $s$ such that $m_j^{k d} = \int_0^1\lambda^k\,h_{s, j}(\lambda)\,d\lambda, \: k=0,\dots, s, \: j=1,2$.
\if{
\[ m_j^{k d} = \int_0^1\lambda^k\,h_{s, j}(\lambda)\,d\lambda \:, \quad k=0,\dots, s \:, \quad j=1,2 \:, \]
}\fi
Importantly, if  $(\m^{s d}_j) = (\m^{s}_j)$ then $h_{s, j}$ minimizes the $L_2$-norm
$\int_0^1(h(\lambda)-f^*_j(\lambda))^2d\lambda$ (see~\ref{sec:inverse} for more details).
%
\paragraph*{Computational considerations}
The presented parametric optimization methodology
has a high computational cost mainly due to the size of SDP relaxations (\ref{eq:primalsdp}) and the state-of-the-art
for SDP solvers. Indeed, when the relaxation order $d$ is fixed, the size of the SDP matrices involved in~\eqref{eq:primalsdp} grows like $O((n + 1)^d)$ for Problem~$\Plam^{1}$ and like $O((n + 2)^d)$ for problems $\Plam^{\infty}$ and $\Plam^u$. By comparison, when using a discretization scheme, one has to solve $N$ polynomial optimization problems, each one being solved by programs whose SDP matrix size grows like $O(n^d)$. Section~\ref{sec:bench} compares both methods.

Therefore these techniques are of course limited to problems of modest size involving a small or medium number of variables $n$. 
We have been able to handle non convex problems with about $15$ variables. However when a correlative sparsity pattern is present then one may benefit from a sparse variant of the SDP relaxations for parametric POP which permits to handle problems of much larger size (e.g. with more than $500$ variables);
see e.g.~\cite{Waki06sumsof, Lasserre_convergentsdp-relaxations} for more details.

\if{
\section{Approximating the Pareto Curve}
\label{sec:approx}
\subsection{Reduction to Scalar Parametric POP}
\label{subsec:parampop}
Here, we show that computing the set of Pareto points associated with Problem~$\P$ can be achieved with three different parametric polynomial problems. Recall that the feasible set of Problem~$\P$ is $\S := \{ \x \in \R^n : g_1 (\x) \geq 0, \dots,  g_{m+1} (\x) \geq 0 \}$.
\subsubsection{Method (a): convex sum approximation}
\label{subsubsec:cvx}
Consider the scalar objective function $f(\lambda, \x) := \lambda f_1(\x) + (1 - \lambda) f_2(\x)$, $\lambda \in [0, 1]$. 
Let $\K^{1} := [0, 1] \times \S$. Recall from (\ref{eq:flamx}) that function $f^{1}:[0,1]\to\R$ is the optimal value of
Problem~$\Plam^{1}$, i.e.~$f^{1}(\lambda) = \min_\x \{f(\lambda,\x) \: :\: (\lambda,\x) \in \K^1\}$.
If the set $f(\S) + \R_+^2$ is convex, then one can recover  the Pareto curve 
by computing $f^{1}(\lambda)$, for all $\lambda \in [0, 1]$, see~\cite[Corollary 5.29]{jahn:2010:vector}.
\begin{lemma}
\label{th:fscvx}
Assume that $f(\S) + \R_+^2$ is convex. Then, a point $\overline{\x} \in \S$ belongs to the Pareto curve of Problem~$\P$ if and only if there exists some weight $\lambda \in [0, 1]$ such that $\overline{\x}$ is an optimal solution of Problem~$\Plam^{1}$.
\end{lemma}
\subsubsection{ Method (b): weighted Chebyshev approximation }
\label{subsubsec:cheb}
Reformulating Problem~$\P$ using the Chebyshev norm approach is more suitable when the set $f(\S) + \R_+^2$ is not convex. We optimize the scalar criterion $f(\lambda, \x) :=  \max[\lambda f_1(\x), (1 - \lambda) f_2(\x)]$, $\lambda \in [0, 1]$. In this case, we assume without loss of generality that both $f_1$ and $f_2$ are positive. Indeed, for each $j=1, 2$, one can always consider the criterion $\tilde{f_j} := f_j - a_j$, where $a_j$ is any lower bound of the global minimum of $f_j$ over $\S$. Such bounds can be computed efficiently by solving polynomial optimization problems using an SDP hierarchy, see e.g.~\cite{Lasserre:2001:moments}.
In practice, we introduce a lifting variable $\omega$ to represent the $\max$ of the objective function. For scaling purpose, we introduce the constant $C := \max (M_1, M_2)$, with $M_j := \max_{\x \in \S} f_j$, $j=1, 2$. 
Then, one defines the constraint set:
\[ \K^{\infty} := \{ (\lambda, \x, \omega) \in \R^{n + 2} : \x \in \S, \lambda \in [0, 1], \lambda f_1 (\x) / C \leq \omega, (1 - \lambda) f_2 (\x) / C \leq \omega  \}  \: , \]
 which leads to the reformulation of $\Plam^{\infty} : f^{\infty}(\lambda) = \min_{\x, \omega} \{ \omega : (\lambda, \x, \omega) \in \K^{\infty} \}$ consistent with (\ref{eq:flamx-infty}). 
The following lemma is a consequence of~\cite[Corollary 11.21 (b)]{jahn:2010:vector}.
\begin{lemma}
\label{th:fsnoncvx}
A point $\overline{\x} \in \S$  belongs to the Pareto curve  of Problem~$\P$ if and only if there exists some weight $\lambda \in [0, 1]$ such that $\overline{\x}$ is an optimal solution of Problem~$\Plam^{\infty}$.
\end{lemma}

\subsubsection{ Method (c): parametric sublevel set approximation }
\label{subsubsec:L1}
Here, we use an alternative method inspired by~\cite{Gorissen2012319}. Problem~$\P$ can be approximated using the criterion $f_2$ as the objective function and the constraint set
\[
\K^u := \{(\lambda, \x) \in [0, 1] \times \S : (f_1 (\x) - a_1) / (b_1 - a_1) \leq \lambda\},
\] 
 which leads to the parametric POP $\Plam^u : f^u (\lambda) = \min_{\x} \{f_2 (\x) :  (\lambda, \x) \in \K^u\}$
which is consistent with (\ref{eq:flamx-u}), and such that $f^u(\lambda)=f^*_2(\lambda)$ for all $\lambda \in [0,1]$, 
with $a_1$ and $b_1$ as in~\eqref{eq:interval1}.
\begin{lemma}
\label{th:f2u}
Suppose that  $\overline{\x} \in \S$ is an optimal  solution of Problem~$\Plam^u$, with $\lambda\in[0,1]$. Then $\overline{\x}$
belongs to the Pareto curve of Problem~$\P$.
\end{lemma}
\begin{proof}
\label{pr:f2u}
Suppose that there exists $\x \in \S$ such that $f_1(\x) < f_1 (\overline{\x})$ and $f_2(\x) < f_2 (\overline{\x})$. Then $\x$ is feasible for Problem~$\Plam^u$ (since $(f_1(\x) - a_1) / (b_1 - a_1) \leq \lambda$) and $f_2 (\overline{\x}) \leq f_2 (\x)$, which leads to a contradiction.
\end{proof}
\if{
\begin{remark}
\label{rk:f2u}
Solving Problem~$\P_u$, $u \in [a_1, b_1]$ does not allow to recover all the weakly EP optimal points. Indeed, ...
\end{remark}
}\fi

\subsection{A Hierarchy of Semidefinite Relaxations}
\label{subsec:sdp}
Notice that the three problems~$\Plam^{1}$, $\Plam^{\infty}$ and $\Plam^u$ are particular instances of the generic parametric optimization problem $f^*(y) := \min_{(y, \x) \in \K} f(y, \x)$. The feasible set $\K$ (resp.~the objective function $f^*$) corresponds to $\K^{1}$ (resp.~$f^{1}$) for Problem~$\Plam^{1}$, $\K^{\infty}$ (resp.~$f^{\infty}$) for Problem~$\Plam^{\infty}$ and $\K^u$ (resp.~$f^u$) for Problem~$\Plam^u$.

We write $\K := \{ (y, \x) \in \R^{n' + 1} : p_1(y, \x) \geq 0, \dots,  p_{m'}(y, \x) \geq 0 \}$. Note also that $n' = n$ (resp.~$n' = n + 1$) when considering Problem~$\Plam^{1}$ and Problem~$\Plam^u$ (resp.~Problem~$\Plam^{\infty}$).

Let $\Mcal(\K)$ be the space of probability measures supported on $\K$.
The function $f^*$ is well-defined because $f$ is a polynomial and $\K$ is compact.
Let $\a=(a_k)_{k \in \N}$, with $a_k=1/(k+1)$, $\forall k \in \N$ and consider the optimization problem:
\begin{equation}
\label{eq:infsdp}
\Pcal:\quad \left\{
\begin{array}{rll}
\rho := \displaystyle\min_{\mu \in \Mcal (\K)} &\displaystyle\int_{\K} f(y,\x) \, d\mu(y,\x)\\
\mbox{s.t.}&\displaystyle\int_{\K} y^k\,d\mu(y,\x) = a_k,\quad k \in \N \: . \\
\end{array} \right.
\end{equation}
\begin{lemma}
\label{th:1}
The optimization problem $\Pcal$ has an optimal solution $\mu^*\in \Mcal(\K)$ and
if $\rho$ is as in~\eqref{eq:infsdp} then 
\begin{equation}
\label{eq:lem1-1}
\rho = \int_{\K}  f(y,\x) \,d\mu^*\,=\, \displaystyle\int_0^1 f^* (y) \, d y \: .
\end{equation}
Moreover, suppose that for almost all $y \in [0,1]$, the parametric optimization problem 
$f^*(y)=\min_{(y, \x) \in \K} f(y, \x)$ has a unique global minimizer $\x^*(y)$ and let $f_j^* : [0,1] \to \R$ be the function $y \mapsto f_j^*(y):=f_j(\x^*(y))$, $j=1,2$. Then for Problem~$\Plam^{1}$,
\begin{equation}
\label{eq:lem1-2}
\rho  = 
\displaystyle\int_0^1 \lambda f_1^*(\lambda) + (1 - \lambda) f_2^*(\lambda) \, d \lambda \: ,
\end{equation}
for Problem~$\Plam^{\infty}$,
\begin{equation}
\label{eq:lem1-3}
\rho = 
\displaystyle\int_0^1 \max \{\lambda f_1^*(\lambda), (1 - \lambda) f_2^*(\lambda) \}\, d \lambda \: ,
\end{equation}
and for Problem~$\Plam^u$,
\begin{equation}
\label{eq:lem1-4}
\rho = 
\displaystyle\int_0^1 f_2^*(\lambda) \, d\lambda \: .
\end{equation}
\end{lemma}
\begin{proof}
The proof of~\eqref{eq:lem1-1} follows from~\cite[Theorem 2.2]{Lasserre:2010:JMA} with $y$ in lieu of $\y$. 
Now, consider the particular case of Problem~$\Plam^{1}$. If $\Plam^{1}$
has a unique optimal solution $\x^*(\lambda) \in \S$ for almost all $\lambda \in [0,1]$ then
$f^*(\lambda) = \lambda f_1^*(\lambda) + (1 - \lambda)f_2^*(\lambda)$ for almost all
$\lambda \in [0,1]$, and so~\eqref{eq:lem1-2} follows. The proofs of~\eqref{eq:lem1-3} and~\eqref{eq:lem1-4} are similar.  
\end{proof}

We set $p_0 := 1$,  $v_l := \lceil \deg p_l  / 2 \rceil$, $l=0,\dots,m'$ and $d_0 := \max (\lceil d_1  / 2 \rceil, \lceil d_2  / 2\rceil, v_1, \dots, v_{m'})$.
Then, consider the following semidefinite relaxations for $d \geq d_0$:
\begin{equation}
\label{eq:primalsdp}
\left\{\begin{array}{rl}
\min_\z &L_\z(f)\\
\mbox{s.t.}&\M_d(\z)\succeq0,\M_{d-v_l}(p_l\,\z)\succeq0, \: l=1,\ldots,m' \: ,\\
&L_\z(y^k) =a_k, \quad k=0,\dots, 2 d \: .
\end{array}\right.
\end{equation}
\begin{lemma}
\label{lem2}
Assume that for almost all $y \in [0,1]$, the parametric optimization problem $f^*(y)=\min_{(y, \x) \in \K} f(y, \x)$ a unique global minimizer $\x^*(y)$, and let
$\z^d=(z^d_{k\alpha})$, $(k,\alpha)\in\N^{n+1}_{2d}$, be an optimal solution of~\eqref{eq:primalsdp}.
Then
\begin{equation}
\label{eq:lem2-1}
\lim_{d\to\infty}z^d_{k\alpha}= \int_0^1 y^k\,(\x^*(y))^\alpha d y \: .
\end{equation}
In particular, for $s \in \N$,
\begin{equation}
\label{eq:lem2-2}
m^k_j:=\lim_{d\to\infty}\sum_{\alpha}f_{j\alpha}z^d_{k\alpha}=\int_0^1 y^k\,f^*_j(y)\,d y \:,  \quad k=0, \dots,s \:, \quad j=1,2 \: .
\end{equation}
\end{lemma}
\begin{proof}
Let $\mu^*\in \Mcal(\K)$ be an optimal solution of problem $\Pcal$. From
\cite[Theorem 3.3]{Lasserre:2010:JMA},
\[\lim_{d\to\infty}z^d_{k\alpha} = \int_{\K} y^k\x^\alpha\,d\mu^*(y,\x) = \int_0^1 y^k \,(\x^*(y))^\alpha d y \: ,\]
which is~\eqref{eq:lem2-1}. Next, from (\ref{eq:lem2-1}), one has for $s \in \N$:
\begin{eqnarray*}
\lim_{d\to\infty}\sum_{\alpha}f_{j\alpha}z^d_{k\alpha}&=&\int_{\K}y^k f_j(\x) \, d\mu^*(y,\x)\\
&=&\int_0^1 y^k\,f_j(\x^*(y))\,d y
\,=\,\int_0^1 y^k\,f^*_j(y)\,d y \:, \quad k=0, \dots, s \:, \quad j=1,2 \:,
\end{eqnarray*}
which is~\eqref{eq:lem2-2}.
\end{proof}
The dual of the SDP~\eqref{eq:primalsdp} reads:

\begin{equation}
\label{eq:dualsdp}
\left\{\begin{array}{rl}
\rho_d^* := & \displaystyle\max_{q, (\sigma_l)}  \displaystyle\int_{0}^1 q (y) \, d y \: (= \displaystyle\sum_{k=0}^{2 d} q_k \, a_k) \:,\\
\mbox{s.t.} & f (y, \x) -q(y)= \sum_{k=0}^{m'} \sigma_l(y, \x) \, p_l (y, \x) \:, \quad \forall y \:, \forall \x  \:,\\
& q \in \R[y]_{2d} , \sigma_l \in \Sigma[y,\x]_{d - v_l}, \: l=0,\ldots,m' \: .\\
\end{array}\right.
\end{equation}

\begin{lemma}
\label{th:dual}
Consider the dual semidefinite relaxations defined in~\eqref{eq:dualsdp}. Then, one has:
\begin{enumerate}[noitemsep,topsep=0pt,label={(\roman*)}]
\item $\rho_d \uparrow \rho$ as $d \to \infty$.
\item Let $q_{2 d}$ be a nearly optimal solution of~\eqref{eq:dualsdp}, i.e., such that 
$\int_0^1q_{2d}(y)dy\geq\rho^*_d-1/d$. Then $q_{2d}$ underestimates $f^*$ over $\S$ and
\[ \lim_{d \to \infty} \int_0^1 | f^* (y) - q_{2 d} (y) | d \mu  \: .\]
\end{enumerate}
\end{lemma}

\begin{proof}
\label{pr:dual}
It follows from~\cite[Theorem 3.5]{Lasserre:2010:JMA}.
\end{proof}

Note that one can directly approximate the Pareto front from below when considering Problem~$\Plam^u$. Indeed, solving the dual SDP~\eqref{eq:dualsdp} yields polynomials that underestimate the function $\lambda \mapsto f_2^*(\lambda)$ over $[0, 1]$.
\begin{remark}
\label{rk:hertog}
In~\cite[Appendix A]{Gorissen2012319}, the authors derive the following relaxation from Problem~$\Plam^u$:
\begin{equation}
\label{eq:dualsdphertog}
\left\{\begin{array}{rl}
 \displaystyle\max_{q \in \R[y]_{d}}  & \displaystyle\int_{0}^1 q (\lambda) \, d\lambda\:,\\
\mbox{s.t.} & f_2 (\x) \geq q(  \frac{f_1(\x) - a_1}{b_1 - a_1} ) \:, \forall \x \in \S \:.\\
\end{array}\right.
\end{equation}
Observe that the set of Pareto points is nonincreasing. Therefore
suppose that in (\ref{eq:dualsdphertog}  
one also imposes that $q$ is nondecreasing. For even degree approximations, the formulation~\eqref{eq:dualsdphertog} is equivalent to
\begin{equation}
\label{eq:dualsdphertogeq}
\left\{\begin{array}{rl}
\displaystyle\max_{q \in \R[y]_{2 d}}  & \displaystyle\int_{0}^1 q (\lambda) \, d\lambda \:,\\
\mbox{s.t.} & f_2 (\lambda) \geq q(\lambda)  \:,  (f_1 (\x) - a_1) / (b_1 - a_1) \leq \lambda \:,\\
& \forall \lambda \in [0, 1]  \:, \forall \x \in \S  \:.
\end{array}\right.
\end{equation}
Thus, our framework is related to~\cite{Gorissen2012319} by observing that~\eqref{eq:dualsdp} is a strengthening of~\eqref{eq:dualsdphertogeq}.
\end{remark}


When using the reformulations $\Plam^{1}$ and $\Plam^{\infty}$, computing the Pareto curve is computing (or at least providing good approximations)
of the functions $f^*_j:[0,1]\to\R$ defined above, and
we consider this problem as an {\it inverse problem from generalized moments}.

$\bullet$ For any fixed $s \in \N$, we first compute approximations $\m^{s d}_j = (m_j^{k d})$, $k = 0, \dots, s$, $d \in \N$, of the generalized moments
\begin{equation}
\label{eq:mom}
m^k_j\,\:=\,\int_0^1\lambda^k\,\underbrace{f^*_j(\lambda)\,d\lambda}_{d\mu_j(\lambda)} \:, \quad k=0,\dots, s \:, \quad j=1,2 \:,
\end{equation}
with the convergence property $(\m^{s d}_j)  \to \m^s _j$ as $d\to\infty$, for each $j=1,2$.

$\bullet$ Then we solve the inverse problem: given a (good) approximation  $(\m^{s d}_j)$ of $ \m^s _j$,
find a polynomial $h_{s, j}$ of degree at most $s$ such that 
\[ m_j^{k d} = \int_0^1\lambda^k\,h_{s, j}(\lambda)\,d\lambda \:, \quad k=0,\dots, s \:, \quad j=1,2 \:, \]
Importantly, if  $(\m^{s d}_j) = (\m^{s}_j)$ then $h_{s, j}$ minimizes the $L_2$-norm
$\int_0^1(h(\lambda)-f^*_j(\lambda))^2d\lambda$ (see Appendix~\ref{sec:inverse} for more details).
%


\paragraph*{Computational considerations.}
The presented parametric optimization methodology
has a high computational cost mainly due to the size of SDP relaxations (\ref{eq:primalsdp}) and the state-of-the-art
for SDP solvers. Indeed, when the relaxation order $d$ is fixed, the size of the SDP matrices involved in~\eqref{eq:primalsdp} grows like $O((n + 1)^d)$ for Problem~$\Plam^{1}$ and like $O((n + 2)^d)$ for problems $\Plam^{\infty}$ and $\Plam^u$. By comparison, when using a discretization scheme, one has to solve $N$ polynomial optimization problems, each one being solved by programs whose SDP matrix size grows like $O(n^d)$. Section~\ref{sec:bench} compares both methods.

Therefore these techniques are of course limited to problems of modest size involving a small or medium number of variables $n$. 
We have been able to handle non convex problems with about $15$ variables. However when a correlative sparsity pattern is present then one may benefit from a sparse variant of the SDP relaxations for parametric POP which permits to handle problems of much larger size (e.g. with more than $500$ variables);
see e.g.~\cite{Waki06sumsof, Lasserre_convergentsdp-relaxations} for more details.

}\fi
\section{Numerical Experiments}
\label{sec:bench}
The semidefinite relaxations of problems~$\Plam^{1}$, $\Plam^{\infty}$ and $\Plam^u$ have been implemented in MATLAB, using the Gloptipoly software package~\cite{henrion:hal-00172442}, on an Intel Core i5 CPU ($2.40\, $GHz). 

\subsection{Case 1: \texorpdfstring{$f(\S) + \R_+^2$}{} is convex}
We have considered the following test problem mentioned in~\cite[Example 11.8]{jahn:2010:vector}:
\begin{example}
\label{ex:ex11_4}
Let
\begin{align*}
& g_1 := - x_1^2 + x_2 \:, & & f_1 := -x_1 \:,\\
& g_2 := - x_1 - 2 x_2 + 3 \:,  & & f_2 := x_1 + x_2^2 \:.\\ 
& \S := \{\x \in \R^2 : g_1(\x) \geq 0, g_2(\x) \geq 0 \} \: . & &
\end{align*}
Figure~\ref{fig:SfScvx} displays the discretization of the feasible set $\S$ as well as the image set $f(\S)$.  The weighted sum approximation method (a) being suitable when the set $f(\S) + \R_+^2$ is convex, one reformulates the problem as a particular instance of Problem~$\Plam^{1}$.
\begin{figure}[!ht]
\centering
\subfigure[$\S$]{
\includegraphics[scale=\sizefig]{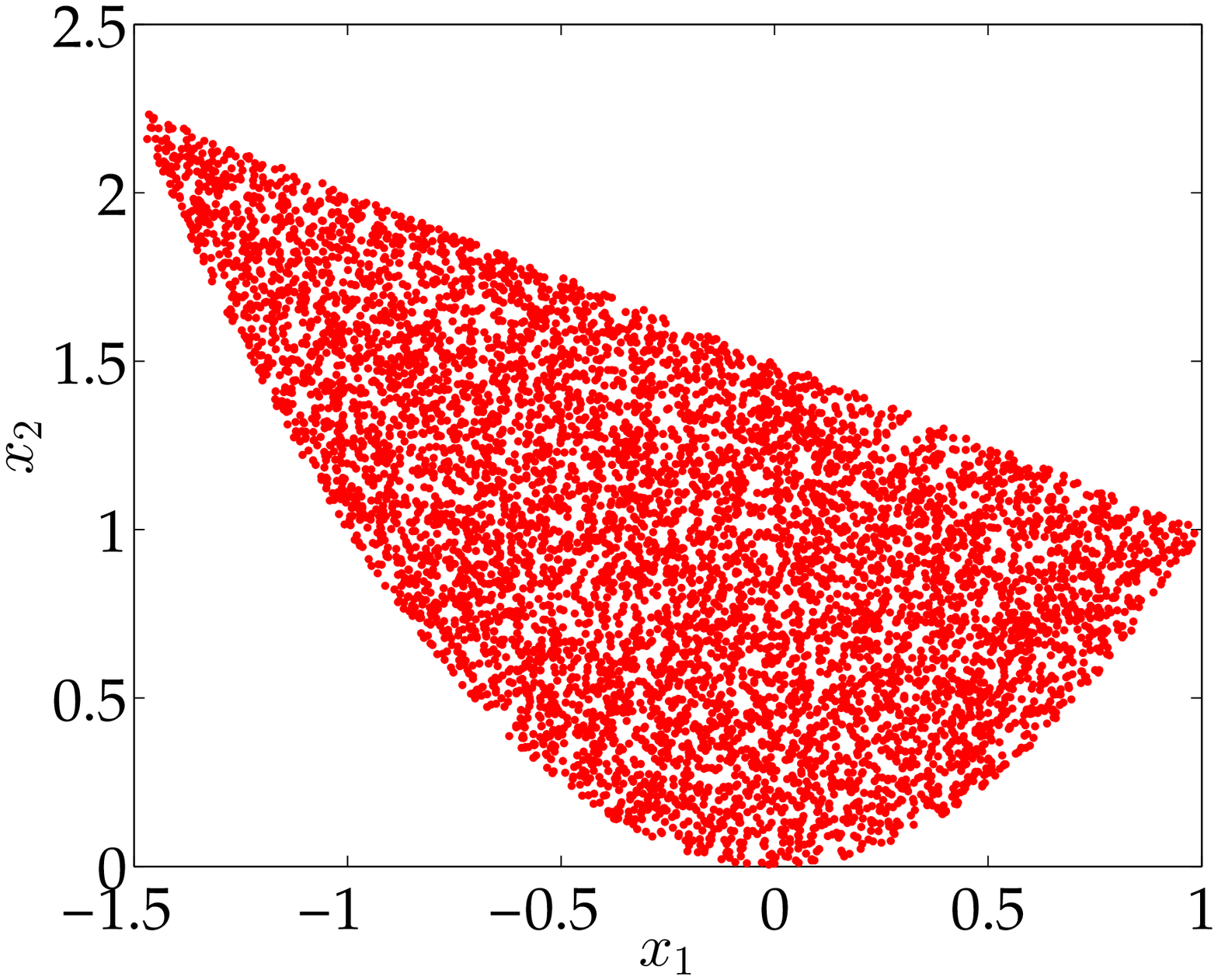}}
\hspace{1cm}
 \subfigure[$f(\S)$]{
\includegraphics[scale=\sizefig]{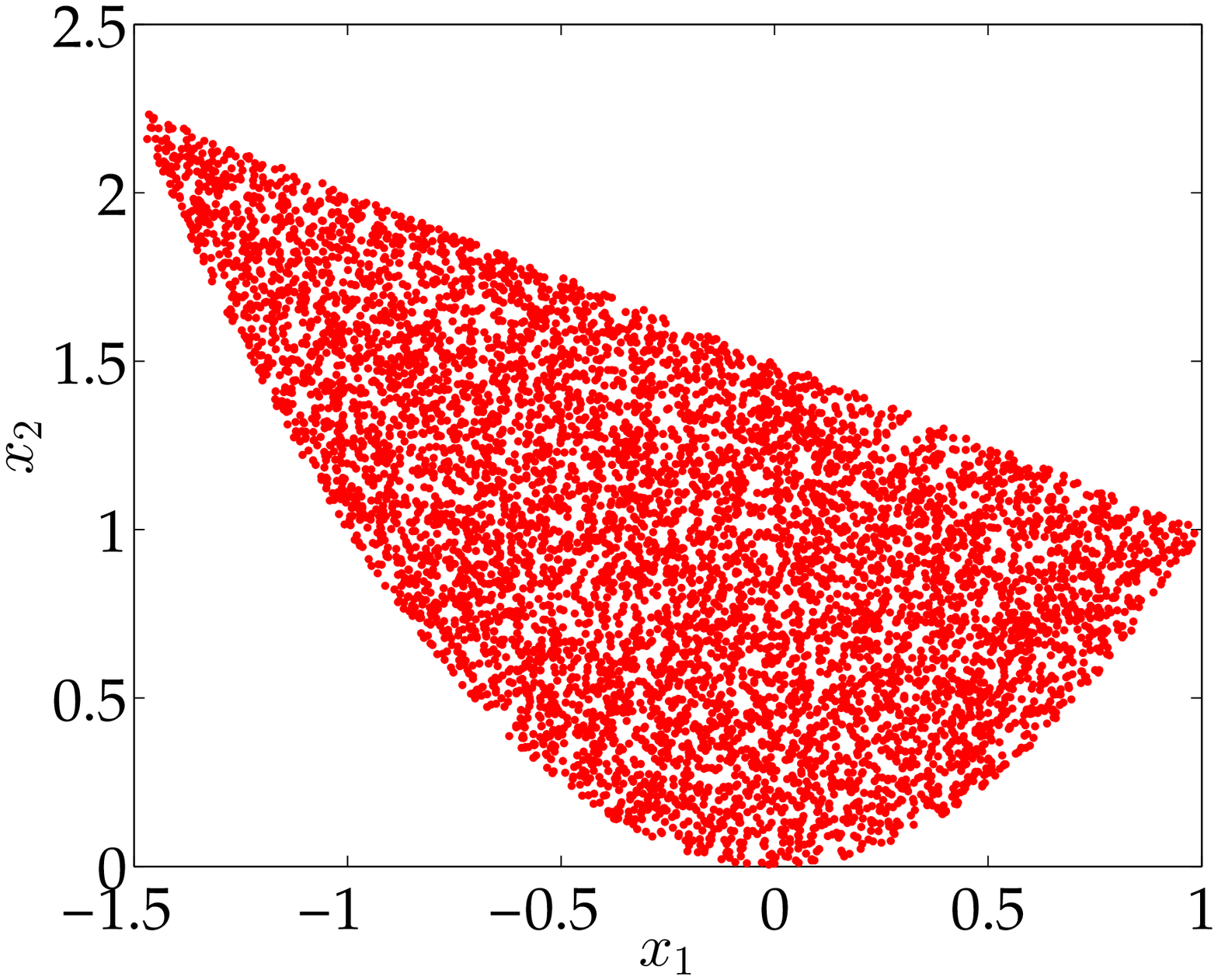}}
 \caption{Preimage and image set of $f$ for Example~\ref{ex:ex11_4}}	\label{fig:SfScvx}
\end{figure}

For comparison, we fix discretization points $\lambda_1, \dots, \lambda_N$ uniformly distributed on the interval $[0, 1]$ (in our experiments, we set $N = 100$).
Then for each $\lambda_i, i = 1, \dots, N$, we compute the optimal value $f^*(\lambda_i)$ of the polynomial optimization problem $\P_{\lambda_i}^{1}$. The dotted curves from Figure~\ref{fig:undercvx} display the results of this discretization scheme. From the optimal solution of the dual SDP~\eqref{eq:dualsdp}  corresponding to our method (a), namely weighted convex sum approximation,  one obtains the degree 4 polynomial $q_4$ (resp. degree 6 polynomial $q_6$) with moments up to order 8 (resp.~12), displayed on Figure~\ref{fig:undercvx}~(a) (resp.~(b)). One observes that $q_4 \leq f^*$ and $q_6 \leq f^*$, which illustrates Lemma~\ref{th:dual} (ii). The higher relaxation order also provides a tighter underestimator, as expected.
\begin{figure}[!ht]
\centering
\subfigure[Degree 4 underestimator]{
\includegraphics[scale=\sizefig]{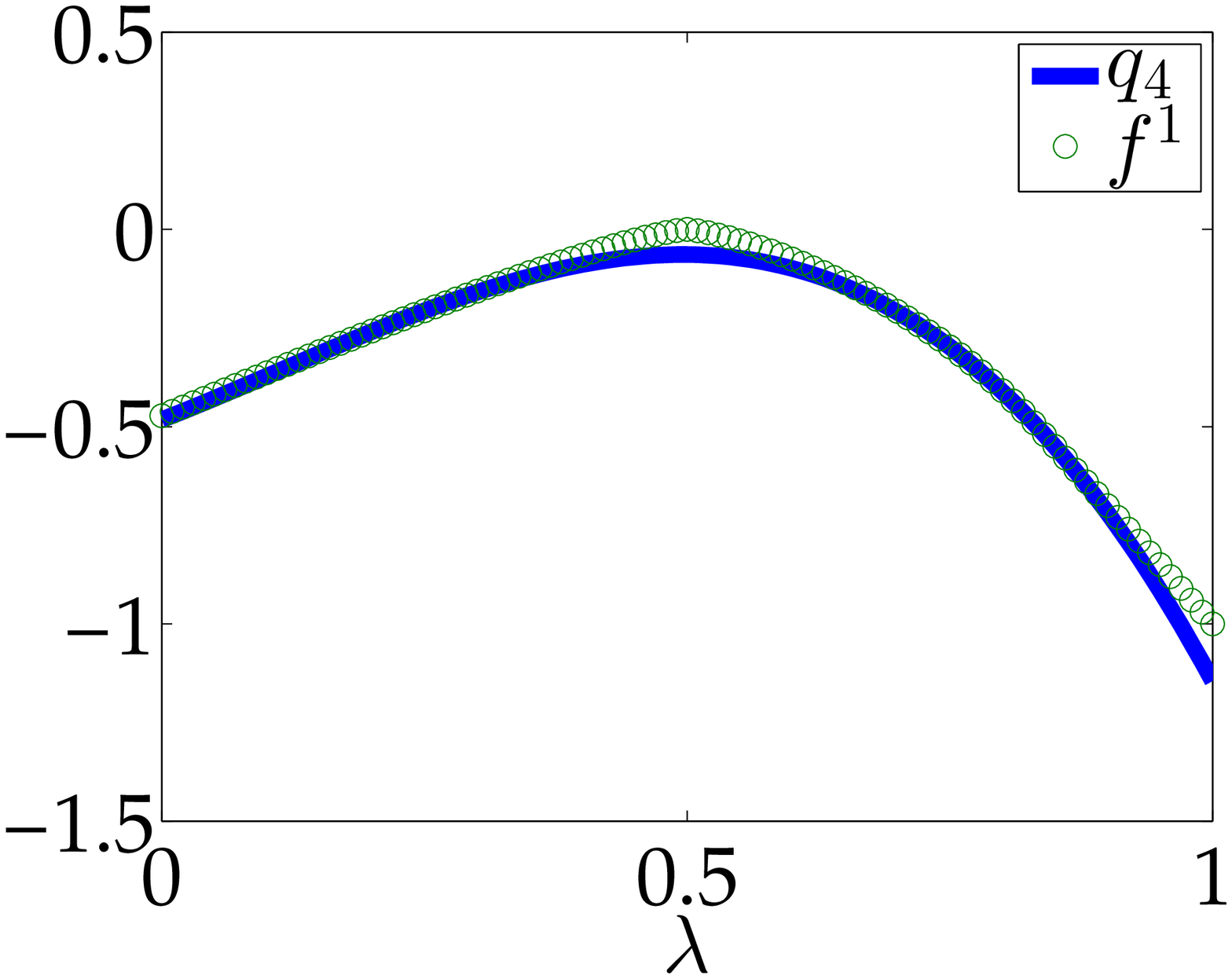}}
\subfigure[Degree 6 underestimator]{
\includegraphics[scale=\sizefig]{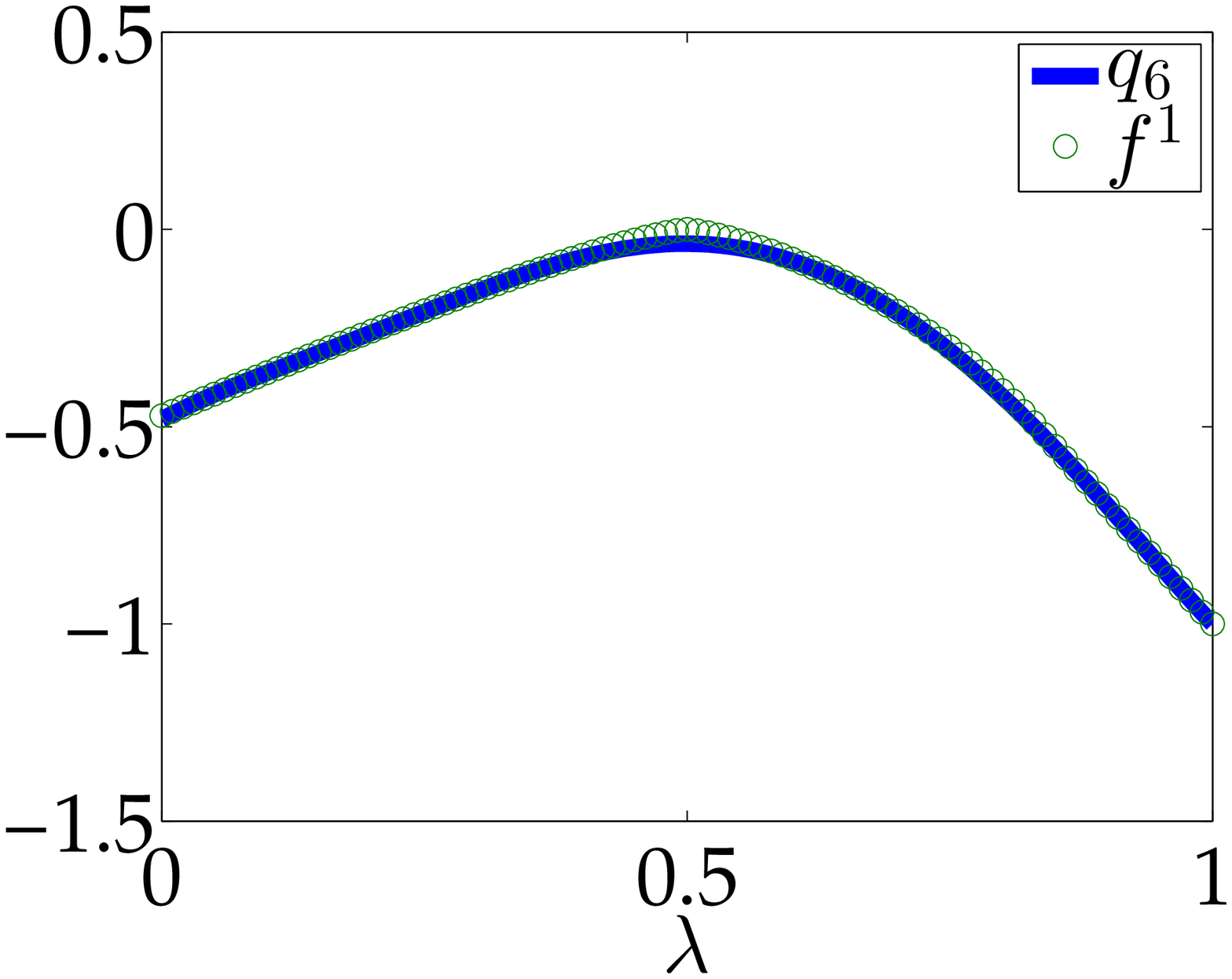}}
 \caption{A hierarchy of polynomial underestimators of the Pareto curve for Example~\ref{ex:ex11_4} obtained by weighted convex sum approximation
(method (a))}	\label{fig:undercvx}
\end{figure}

Then, for each $\lambda_i, i = 1, \dots, 100$, we compute an optimal solution $\x^*(\lambda_i)$ of Problem~$\P_{\lambda_i}^{1}$ and we set $f_{1 i}^* := f_1(\x^*(\lambda_i))$, $f_{2 i}^* := f_2(\x^*(\lambda_i))$. Hence, we obtain a discretization $(f_1^*, f_2^*)$ of the
 Pareto curve,  represented by the dotted curve on Figure~\ref{fig:approxcvx}. The required CPU running time for the corresponding SDP relaxations is $26$sec.

We compute an optimal solution of the primal SDP~\eqref{eq:primalsdp} at order $d = 5$, in order to provide a good approximation of $s  + 1$ moments with $s=4, 6, 8$. Then, we approximate each function $f_j^*$, $j = 1,2$ with a polynomial $h_{s j}$ of degree $s$ by solving the inverse problem from generalized moments (see Appendix~\ref{sec:inverse}). The resulting Pareto curve approximation using degree 4 estimators $h_{4 1}$ and $h_{4 2}$ is displayed on Figure~\ref{fig:approxcvx}~(a). For comparison purpose, higher degree approximations are also represented on Figure~\ref{fig:approxcvx}~(b) (degree 6 polynomials) and Figure~\ref{fig:approxcvx}~(c) (degree 8 polynomials). It consumes only $0.4$sec to compute the two degree 4 polynomials $h_{4 1}$ and $h_{4 2}$, $0.5$sec for the degree 6 polynomials and $1.4$sec for the degree 8 polynomials.

\begin{figure}[!ht]
\centering
\subfigure[Degree 4 estimators]{
\includegraphics[scale=\sizefig]{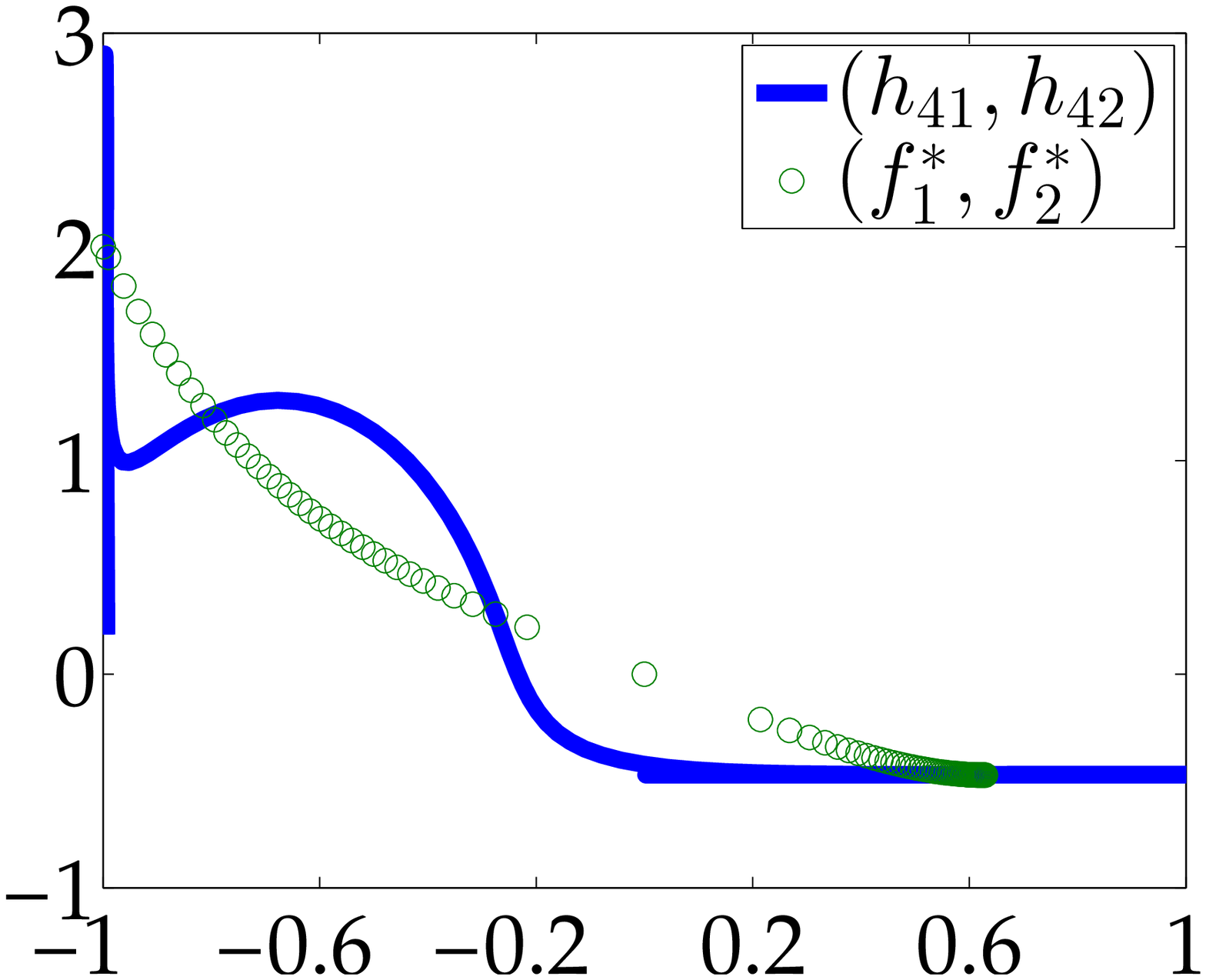}}
\subfigure[Degree 6 estimators]{
\includegraphics[scale=\sizefig]{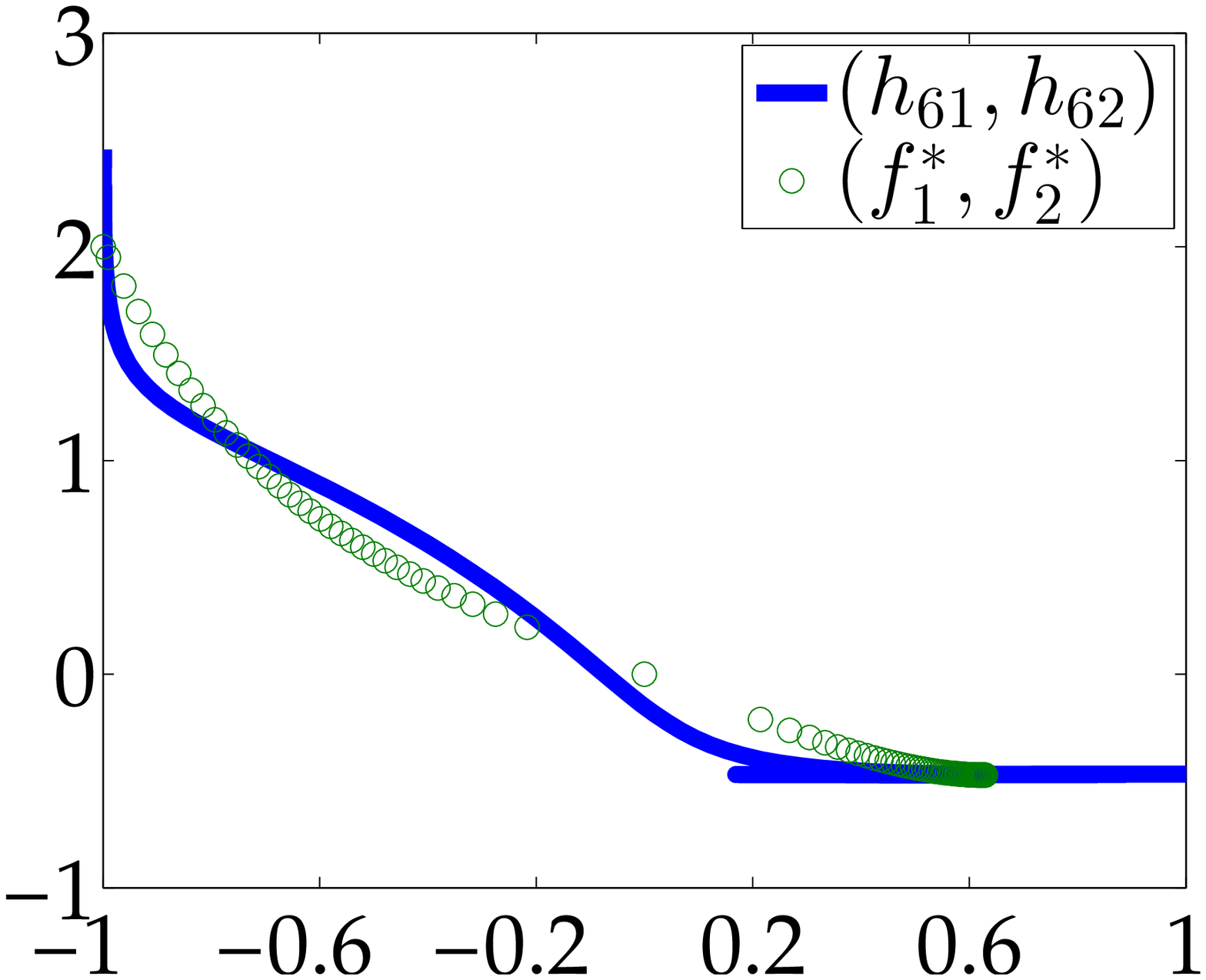}}
\subfigure[Degree 8 estimators]{
\includegraphics[scale=\sizefig]{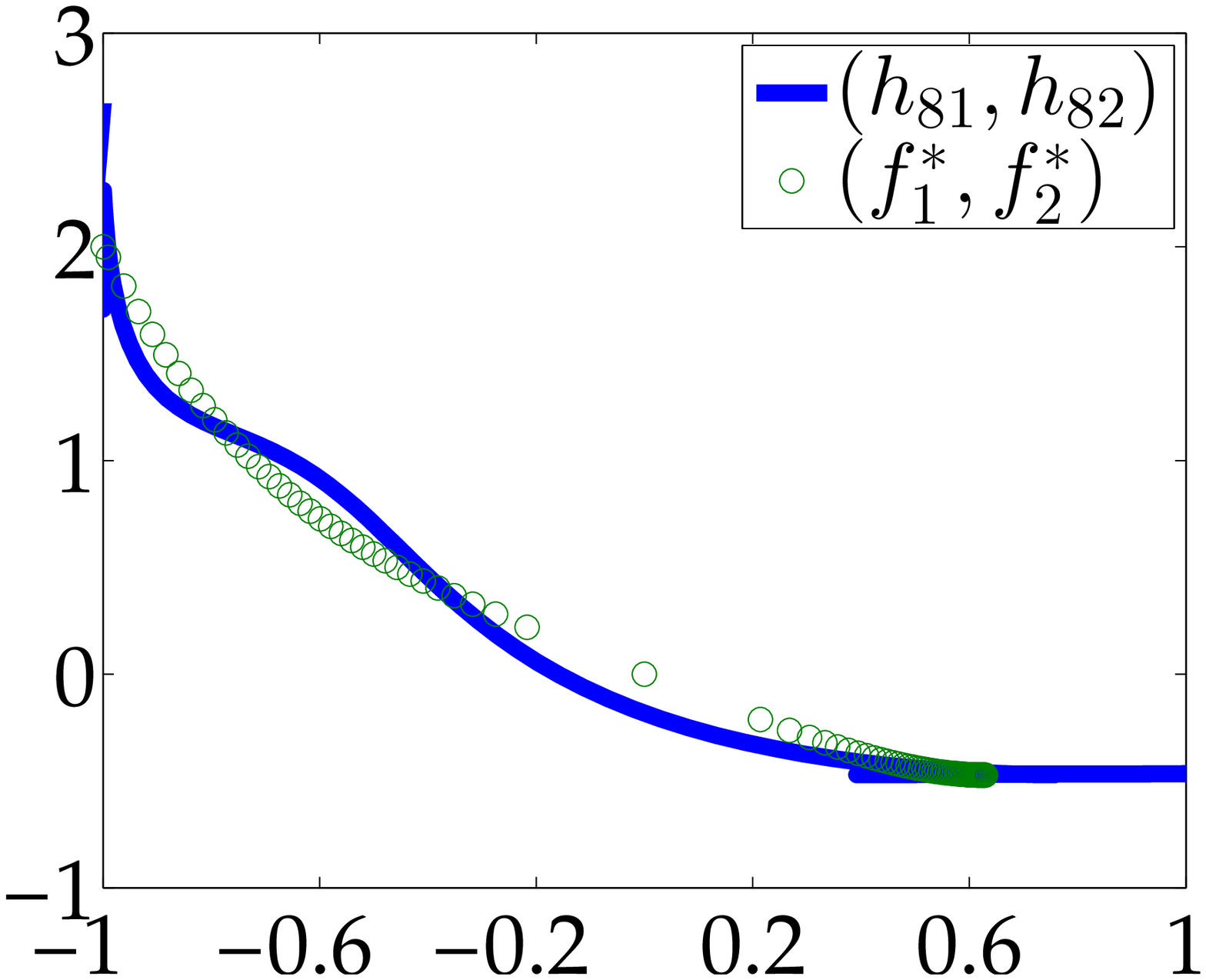}}
 \caption{A hierarchy of polynomial approximations of the Pareto curve for Example~\ref{ex:ex11_4} obtained by the weighted convex sum approximation (method (a))}	\label{fig:approxcvx}
\end{figure}
\end{example}

\subsection{Case 2: \texorpdfstring{$f(\S) + \R_+^2$}{} is not convex}
We have also solved the following two-dimensional nonlinear problem proposed in~\cite{Wilson2001}:
\begin{example}
\label{ex:test4}
Let
\begin{align*}
& g_1 := -(x_1-2)^3/2 - x_2 + 2.5 \:, & & f_1 := \tfrac{(x_1+x_2-7.5)^2}{4} + (x_2-x_1+3)^2 \:,\\
& g_2 := -x_1 - x_2 + 8 (x_2-x_1+0.65)^2 + 3.85 \:,  & & f_2 := 0.4(x_1-1)^2 + 0.4 (x_2-4)^2 \:.\\ 
& \S := \{\x \in [0, 5] \times [0, 3] : g_1(\x) \geq 0, g_2(\x) \geq 0 \} \: . & &
\end{align*}
Figure~\ref{fig:SfSnoncvx} depicts the discretization of the feasible set $\S$ as well as the image set $f(\S)$ for this problem. Note that the Pareto curve is non-connected and non-convex.
\begin{figure}[!ht]
\centering
\subfigure[$\S$]{
\includegraphics[scale=\sizefig]{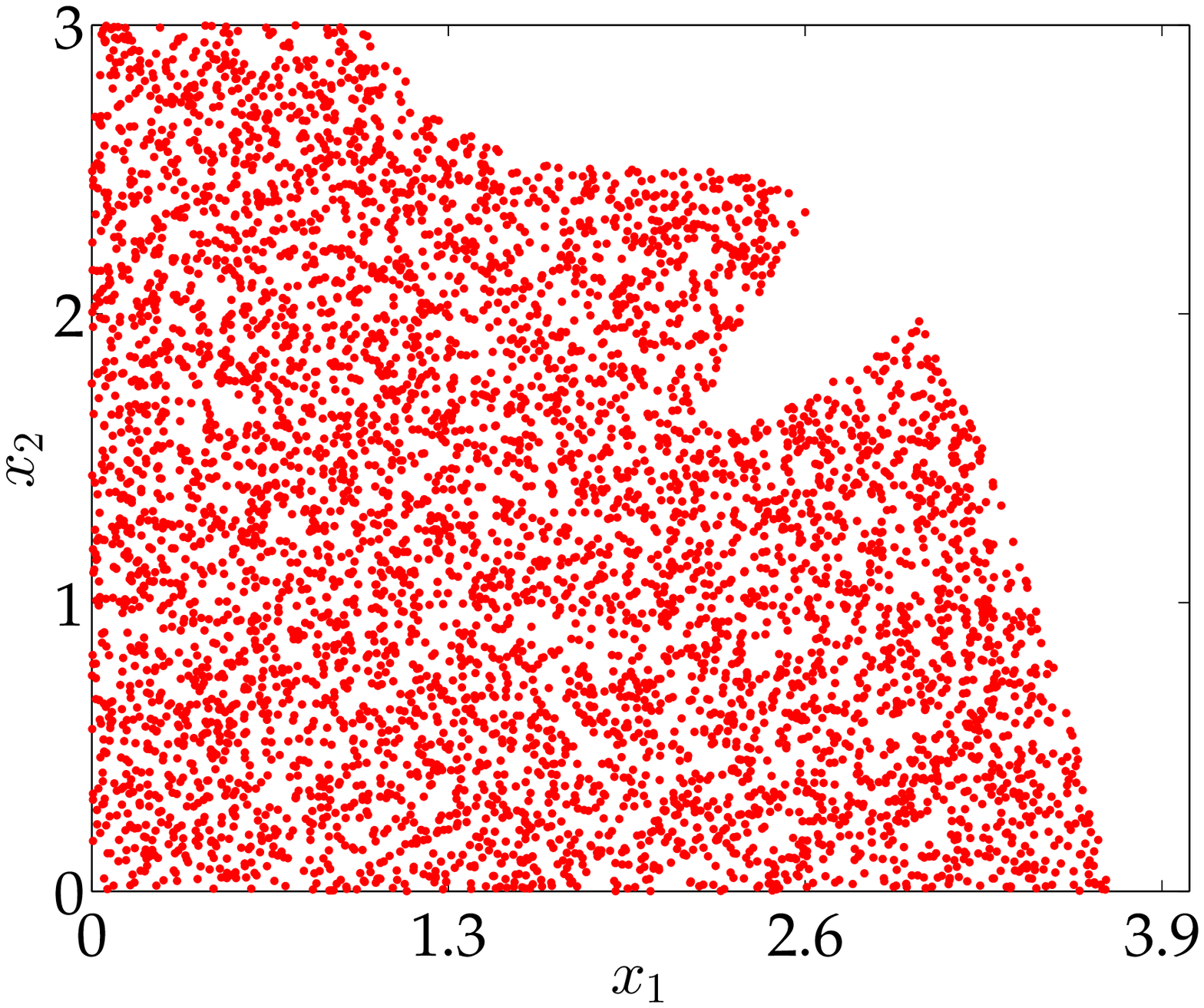}}
\hspace{1cm}
 \subfigure[$f(\S)$]{
\includegraphics[scale=\sizefig]{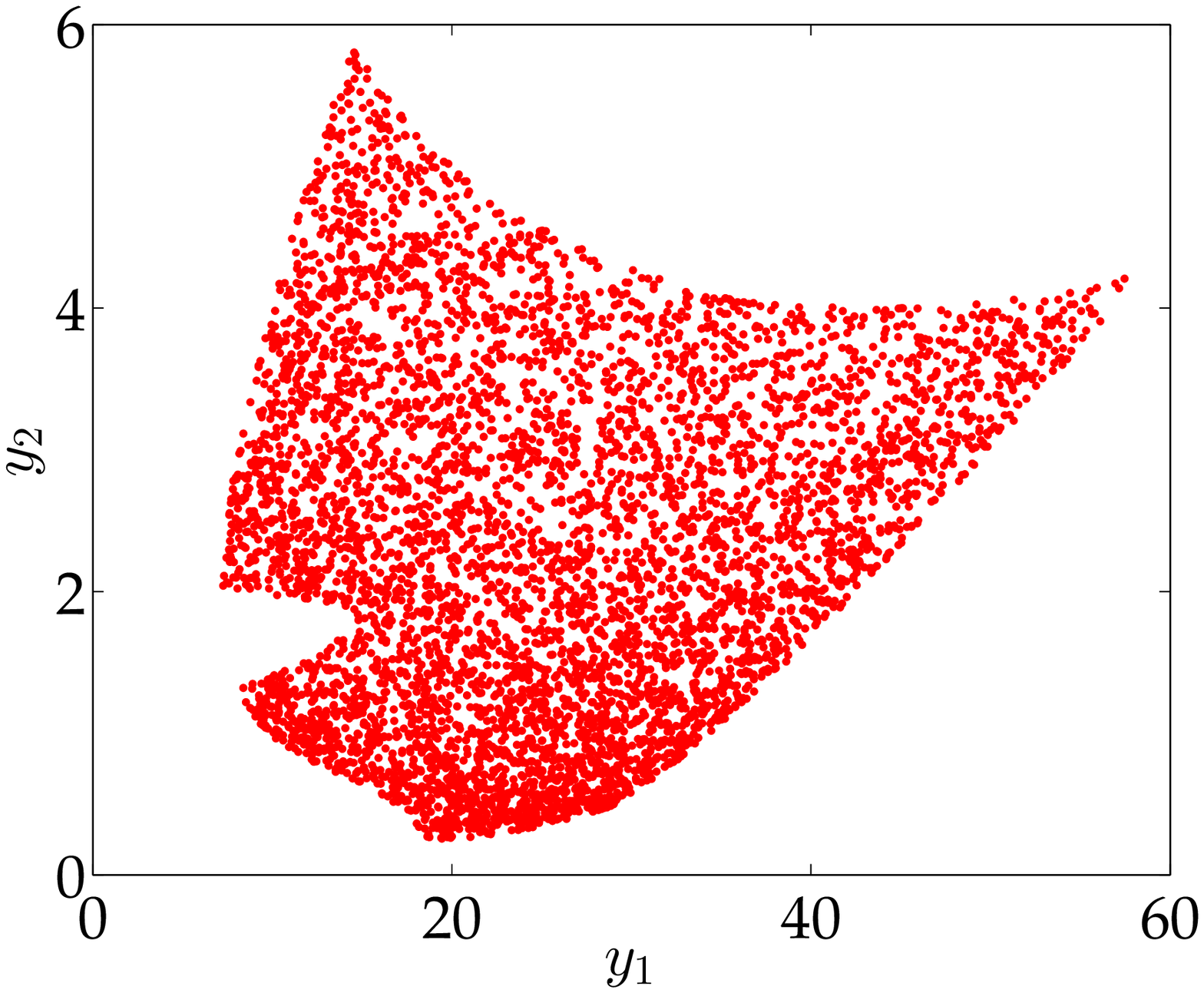}}
 \caption{Preimage and image set of $f$ for Example~\ref{ex:test4}}	\label{fig:SfSnoncvx}
\end{figure}
\end{example}
In this case, the weighted convex sum approximation of method (a) would not allow to properly reconstruct  the Pareto curve, due to the apparent nonconvex geometry of the set $f(\S) + \R_+^2$. Hence we have considered methods (b) and (c).
\paragraph*{Method (b): weighted Chebyshev approximation} 
As for Example~\ref{ex:ex11_4}, one solves the SDP~\eqref{eq:primalsdp} at order $d = 5$ and approximate each function $f_j^*$, $j = 1,2$ using polynomials of degree 4, 6 and 8. The approximation results are displayed on Figure~\ref{fig:approxnoncvx}. Degree 8 polynomials give a closer approximation of the Pareto curve than degree 4 or 6 polynomials. The solution time range is similar to the benchmarks of Example~\ref{ex:ex11_4}. The SDP running time for the discretization is about $3$min. 
The degree 4  polynomials are obtained after $1.3$sec, the degree 6 polynomials $h_{6 1}$, $h_{6 2}$ after $9.7$sec and the degree 8 polynomials after $1$min.
\begin{figure}[!ht]
\centering
\subfigure[Degree 4 estimators]{
\includegraphics[scale=\sizefig]{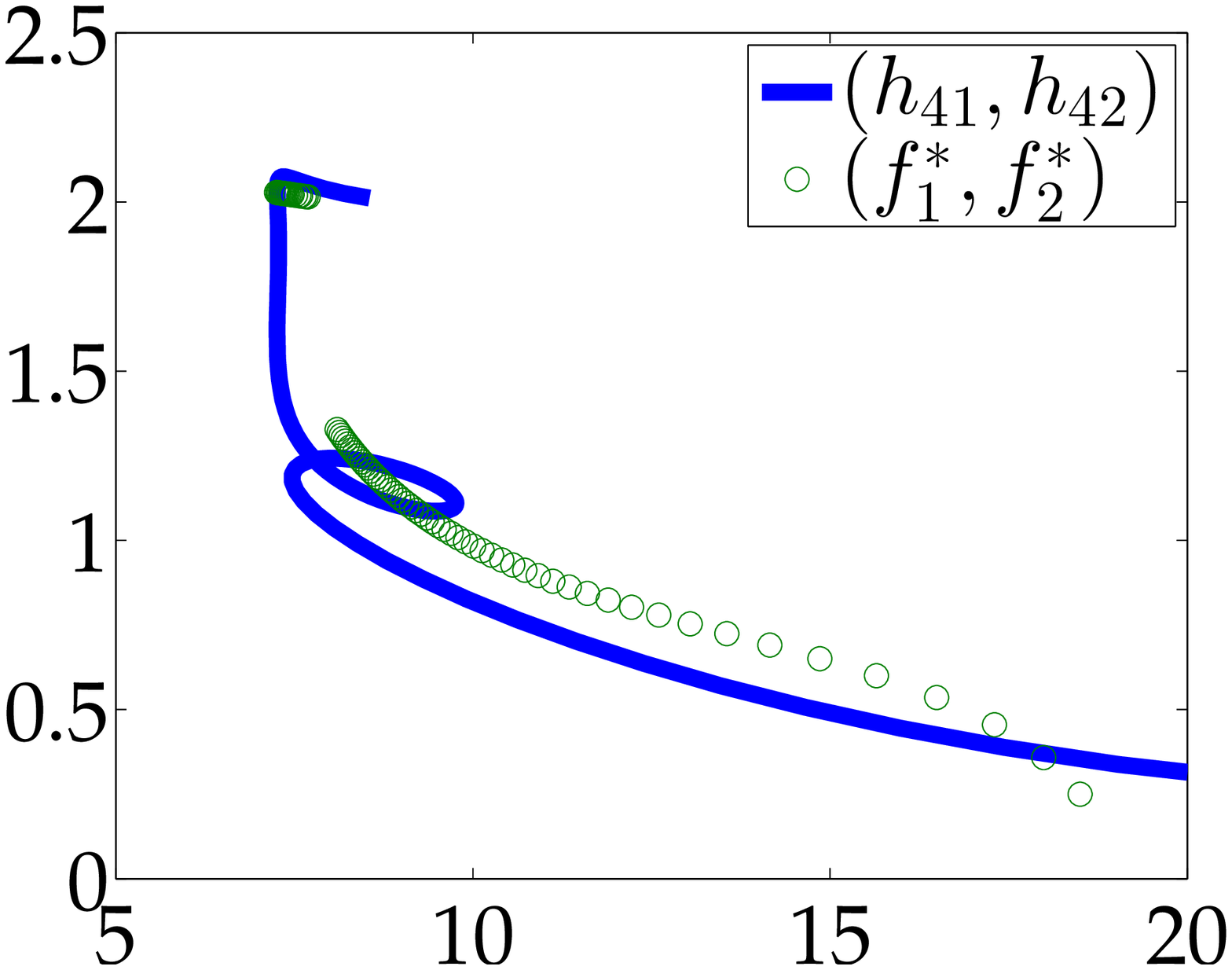}}
\subfigure[Degree 6 estimators]{
\includegraphics[scale=\sizefig]{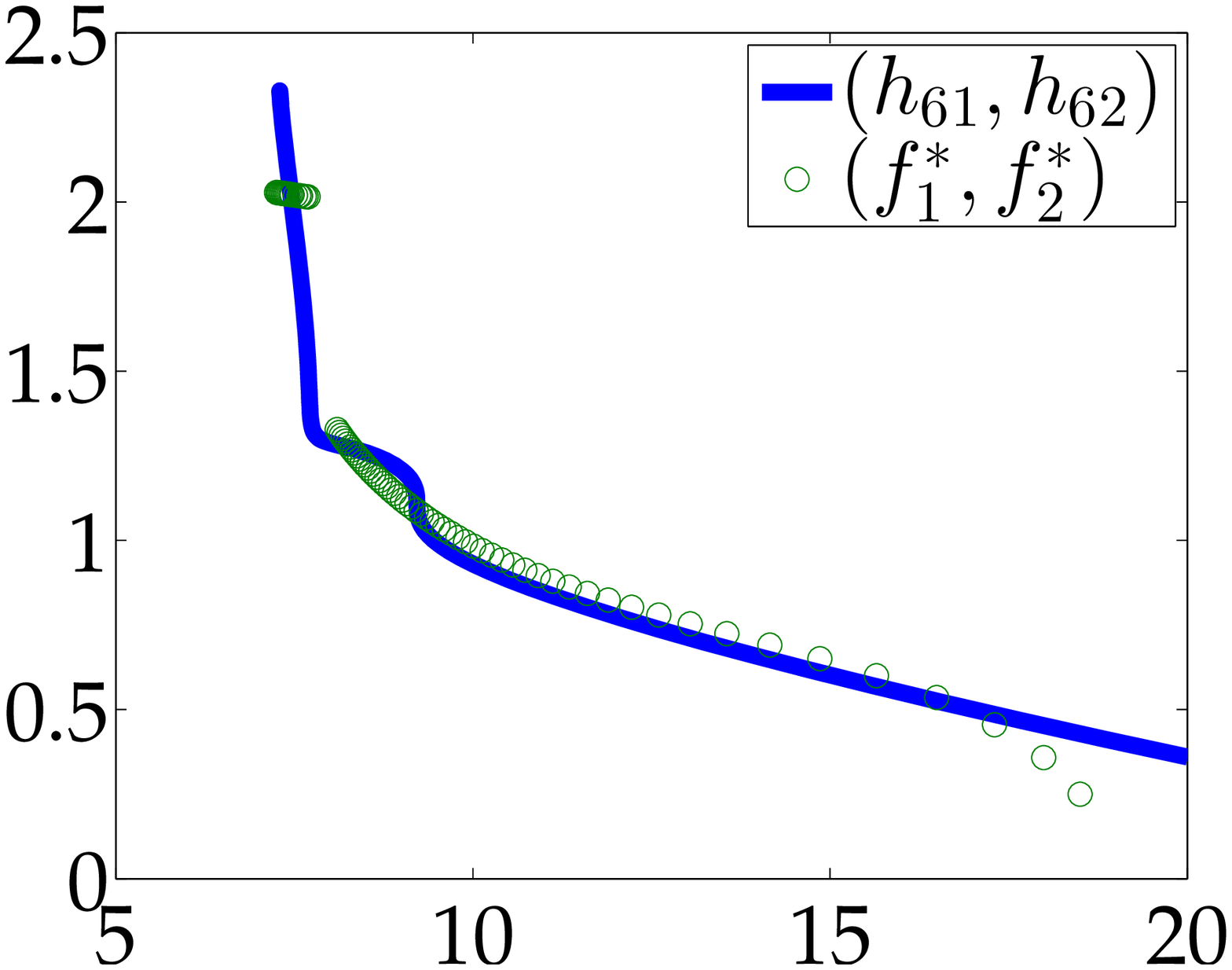}}
\subfigure[Degree 8 estimators]{
\includegraphics[scale=\sizefig]{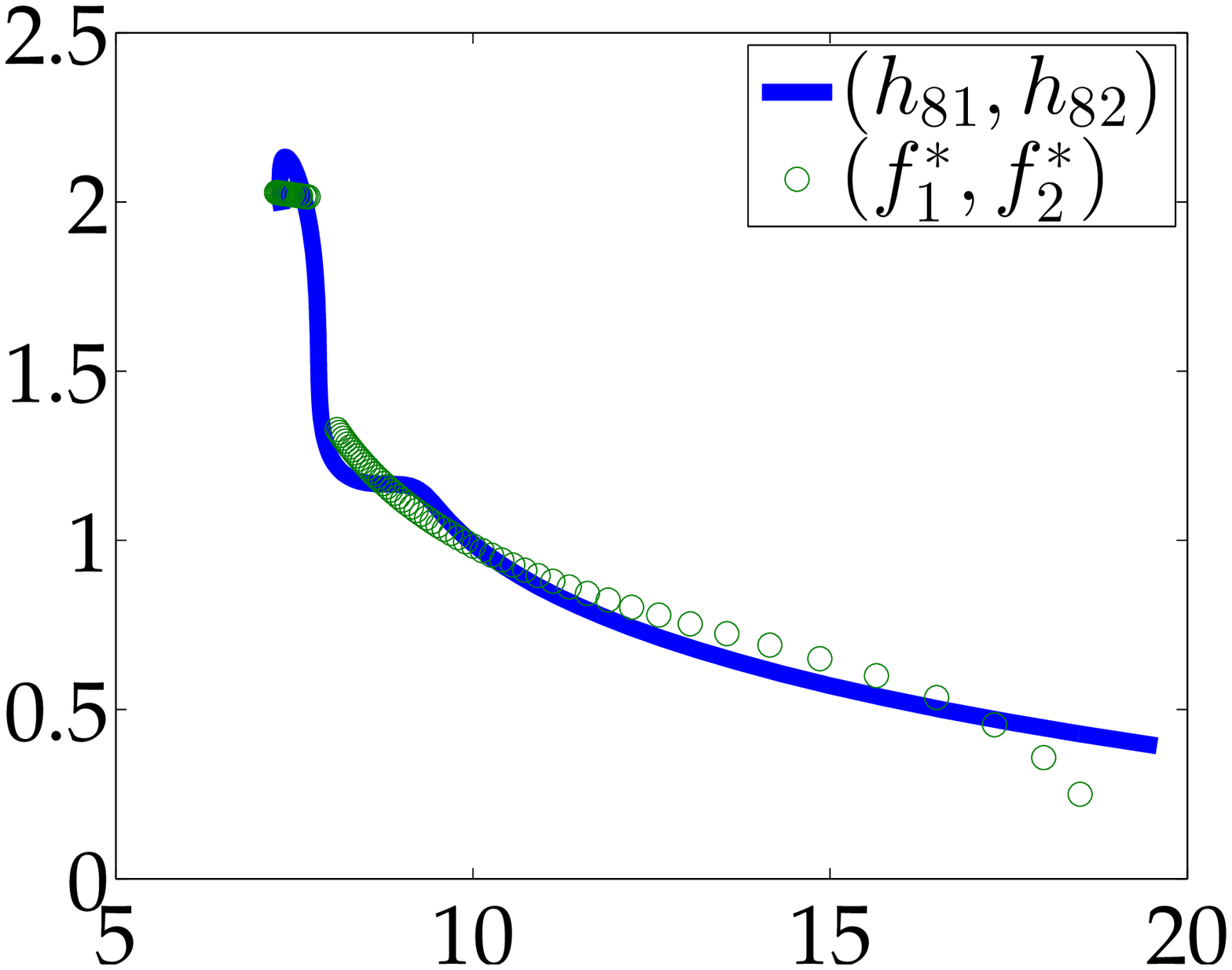}}
 \caption{A hierarchy of polynomial approximations of the Pareto curve for Example~\ref{ex:test4} obtained by the Chebyshev norm approximation (method (b))}	\label{fig:approxnoncvx}
\end{figure}

\paragraph*{Method (c): parametric sublevel set approximation}

Better approximations can be directly obtained by reformulating Example~\ref{ex:test4} as an instance of Problem~$\Plam^u$ and compute the degree $d$ optimal solutions $q_{2d}$ of the dual SDP~\eqref{eq:dualsdp}. Figure~\ref{fig:f2utest4} reveals that 
with  degree 4 polynomials one can already capture 
the change of sign of the Pareto front curvature (arising when the values of $f_1$ lie over $[10, 18]$). Observe also that higher-degree polynomials yield tighter underestimators of the left part of the Pareto front. The CPU time ranges from $0.5$sec to compute the degree 4 polynomial $q_4$, to $1$sec for the degree 6 computation and $1.7$sec  for the degree 8 computation. The discretization of the Pareto front is obtained by solving the polynomial optimization problems $\P^u_{\lambda_i}, i = 1, \dots, N$. The corresponding running time of SDP programs is $51$sec.
\begin{figure}[!ht]
\centering
\subfigure[Degree 4 estimators]{
\includegraphics[scale=\sizefig]{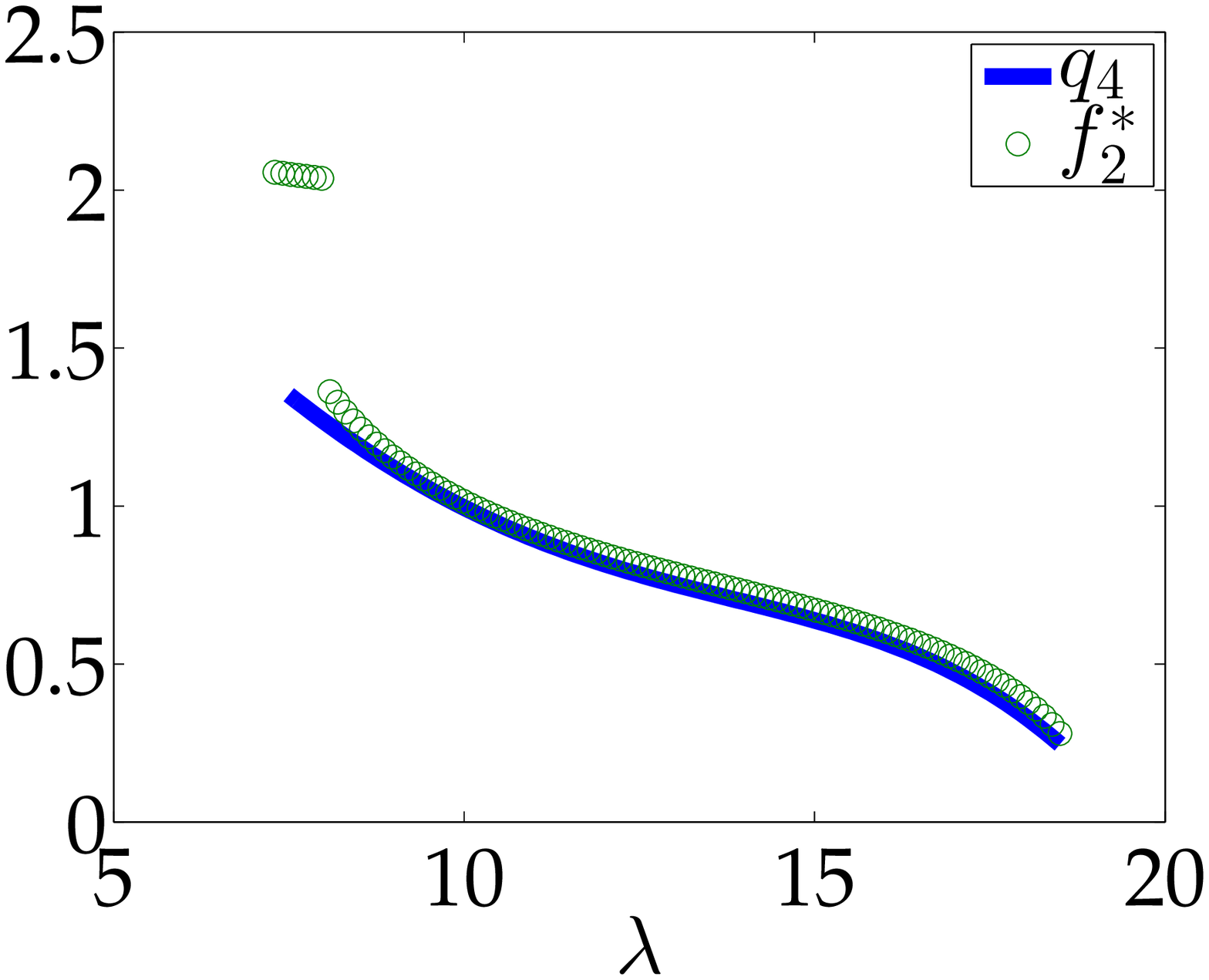}}
\subfigure[Degree 6 estimators]{
\includegraphics[scale=\sizefig]{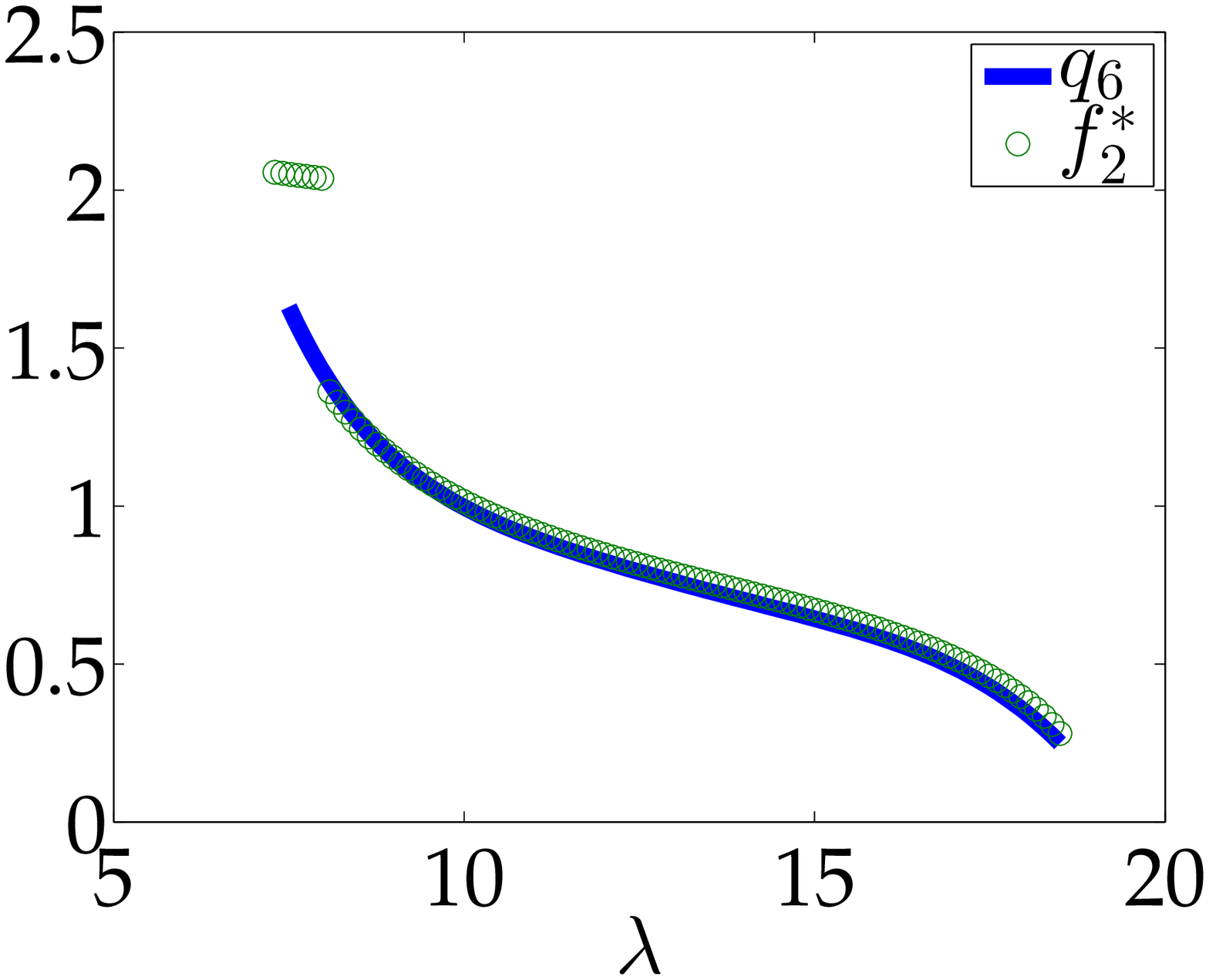}}
\subfigure[Degree  8  estimators]{
\includegraphics[scale=\sizefig]{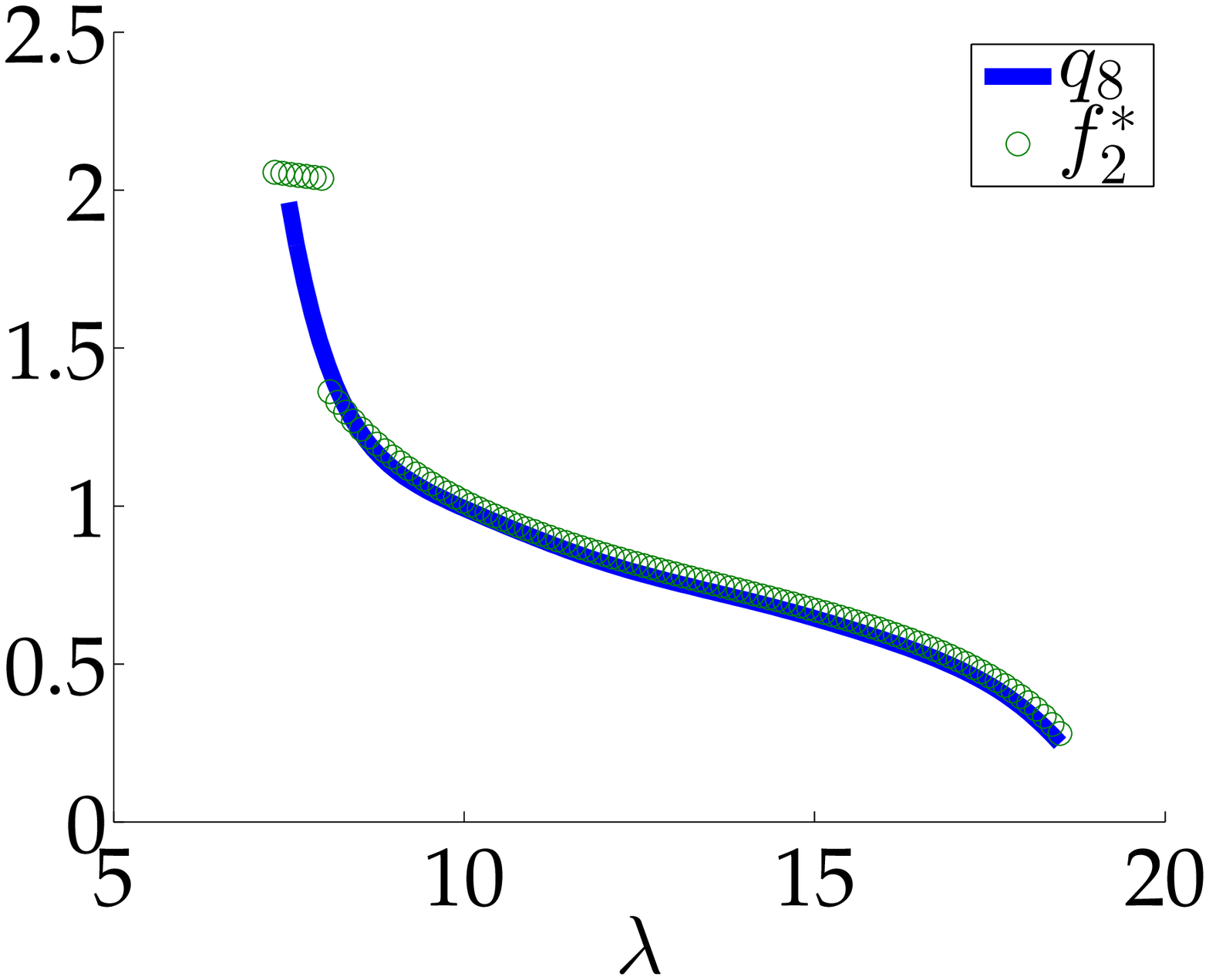}}
\caption{A hierarchy of polynomial underestimators of $\lambda \mapsto f_2^*(\lambda)$ for Example~\ref{ex:test4} obtained by the parametric sublevel set approximation (method (c))}	\label{fig:f2utest4}
\end{figure}

The same approach is used to solve the random bicriteria problem of Example~\ref{ex:rnd}.
\begin{example}
\label{ex:rnd}
Here, we generate two random symmetric real matrices $\Q_1, \Q_2 \in \R^{15 \times 15}$ as well as two random vectors $\q_1, \q_2 \in \R^{15}$. Then we solve the quadratic bicriteria problem $\min_{\x \in [-1, 1]^{15}} \{f_1(\x), f_2(\x) \}$, with $f_j (\x) := \x^\top \Q_j \x / n^2 - \q_j^\top \x / n$, for each $j=1,2$. 

Experimental results are displayed in Figure~\ref{fig:rnd}.
For a 15 variable random instance, it consumes $18$min of CPU time to compute $q_4$ against only $0.5$sec for $q_2$ but the degree 4 underestimator yields a better point-wise approximation of the Pareto curve. The running time of SDP programs is more than $8$ hours to compute the discretization of the front.
\begin{figure}[!ht]
\centering
\subfigure[Degree 2 underestimator]{
\includegraphics[scale=\sizefig]{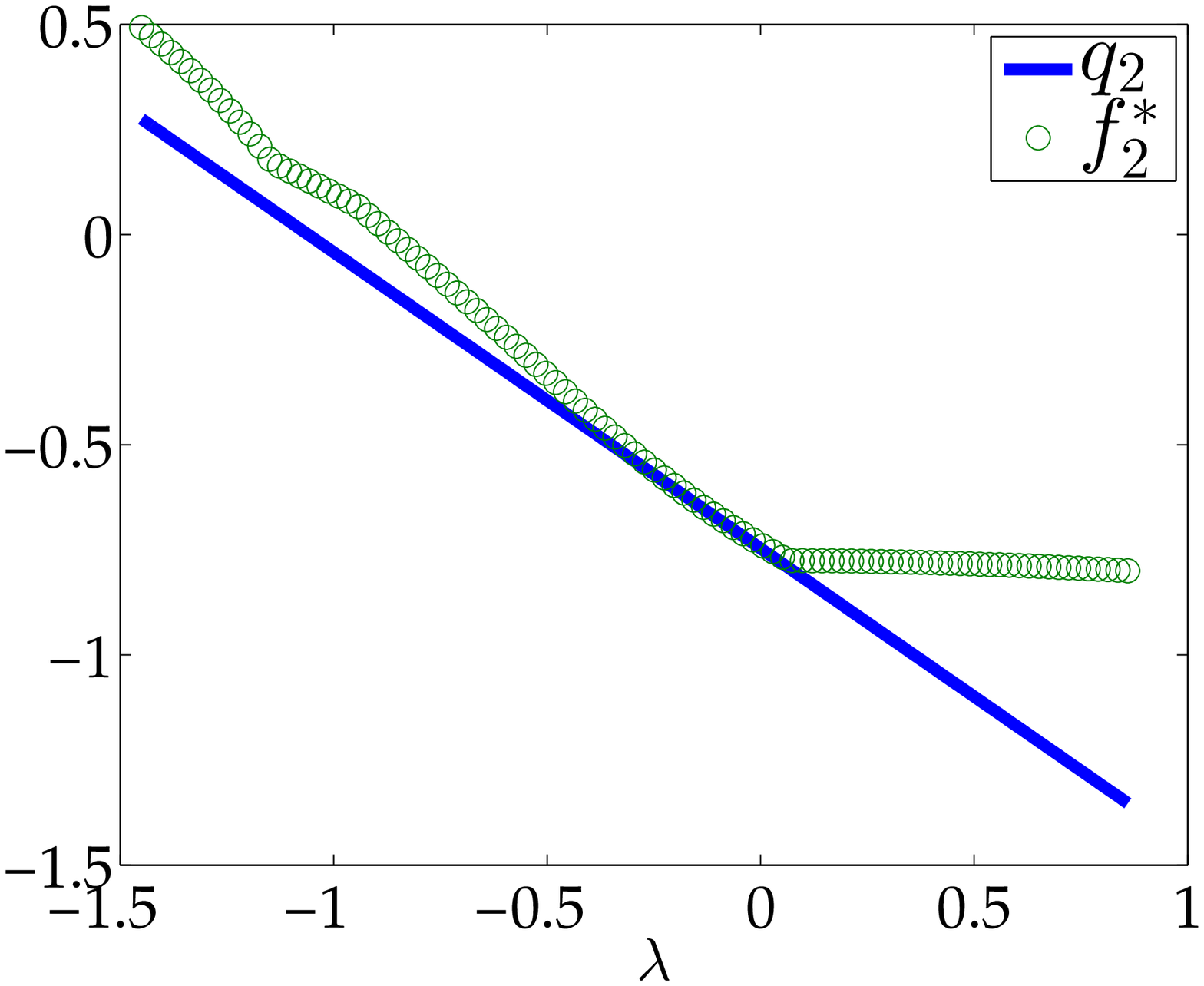}}
\subfigure[Degree 4 underestimator]{
\includegraphics[scale=\sizefig]{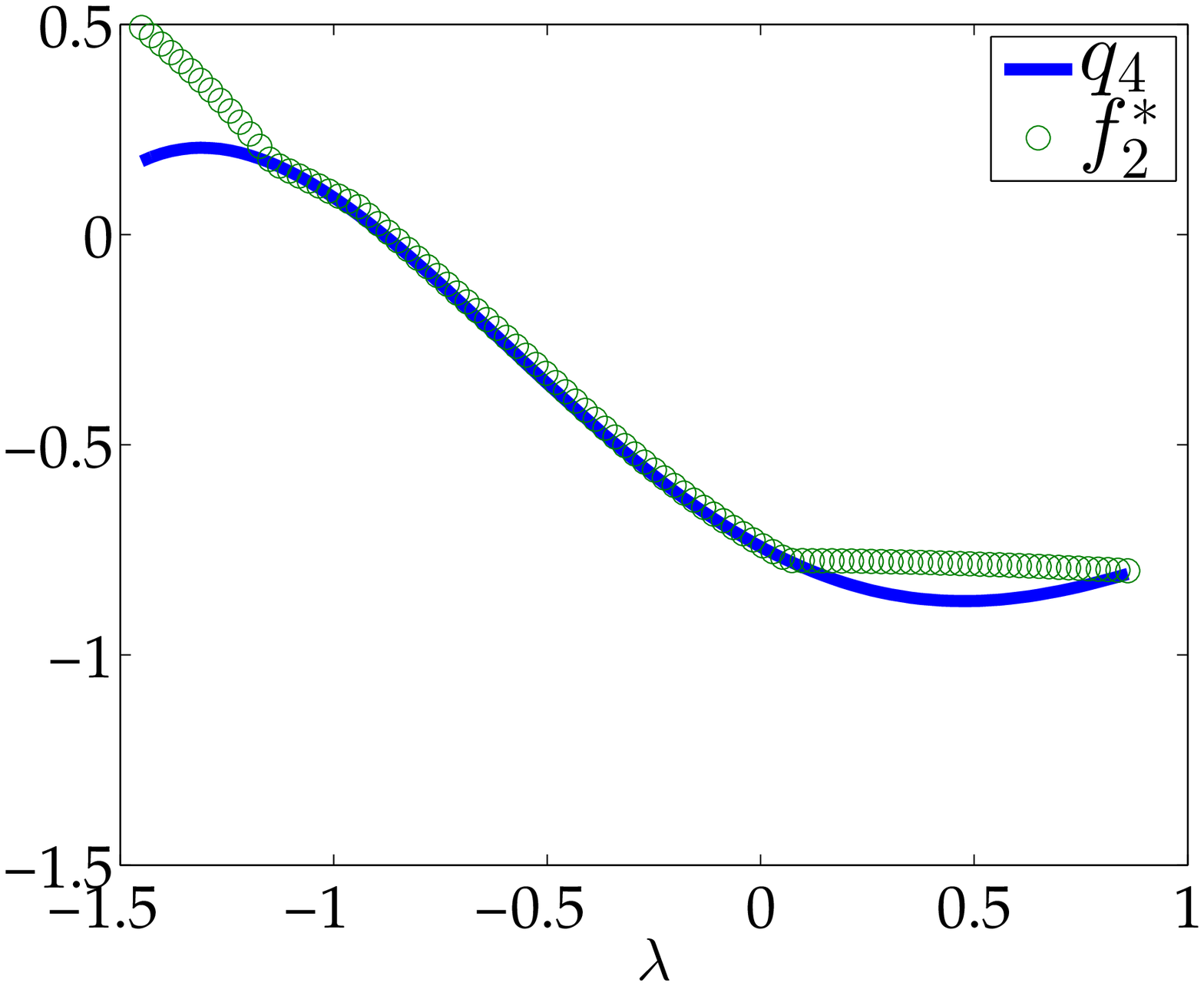}}
 \caption{ A hierarchy of polynomial underestimators of $\lambda \mapsto f_2^*(\lambda)$ for Example~\ref{ex:rnd} obtained by the parametric sublevel set approximation (method (c))}	\label{fig:rnd}
\end{figure}
\end{example}
\section{Conclusion}

The present framework can tackle multicriteria polynomial problems by solving semidefinite relaxations of parametric optimization programs. 
The reformulations based on the weighted sum approach and the Chebyshev approximation allow to recover the Pareto curve, defined here as the set of weakly Edgeworth-Pareto points, by solving an inverse problem from generalized moments. An alternative method builds directly a hierarchy of polynomial underestimators of the Pareto curve.
The numerical experiments illustrate the fact that the Pareto curve can be estimated as closely as desired using semidefinite programming within a reasonable amount of time for problem still of modest size. 
Finally our approach could be extended to higher-dimensional problems by exploiting the system properties such as sparsity patterns or symmetries.
\section*{Acknowledgments}

This work was partly funded by an award of the Simone and Cino del Duca foundation of Institut de France.


                                  
\appendix                                  
\section{Appendix. An Inverse Problem from Generalized Moments}   
\label{sec:inverse}
      
Suppose that one wishes to approximate each function $f^*_j$, $j=1,2$, with a polynomial of degree $s$. One way to do this is to search for $h_j\in\R_{s}[\lambda]$, $j=1,2$, optimal solution of
\begin{equation}
\label{eq:L2}
\min_{h \in\R_{s}[\lambda] } \: \displaystyle\int_0^1(h(\lambda)-f^*_j(\lambda))^2d\lambda\:, \quad j=1,2 \:.
\end{equation}
Let $\H_s \in \R^{(s + 1) \times (s + 1)}$ be the Hankel matrix associated with the moments of the Lebesgue measure on $[0,1]$, i.e.~$\H_s(i,j)=1/(i+j+1)$, $i,j=0,\ldots,s$. 

\begin{theorem}
\label{th:L2}
For each $j=1,2$, let $\m_j^s=(m^k_j)\in\R^{s+1}$ be as in~\eqref{eq:lem2-2}. Then~\eqref{eq:L2} has an optimal solution
$h_{s, j} \in \R_{s}[\lambda]$ whose vector of coefficient $\h_{s, j} \in \R^{s+1}$ is given by:
\begin{equation}
\label{eq:psj}
\h_{s, j} = \H_s^{-1} \m_j^s, \quad j=1,2 \:.
\end{equation}
\end{theorem}
\begin{proof}
Write
\[\int_0^1(h(\lambda) - f^*_j(\lambda))^2d\lambda = \underbrace{\int_0^1h^2d\lambda}_{A}-2\underbrace{\int_0^1h(\lambda)f^*_j(\lambda)d\lambda}_{B}+\underbrace{\int_0^1(f^*_j(\lambda))^2d\lambda}_{C} \: ,\]
and observe that
\[A=\h^\top \H_s \h \:, \quad B= \sum_{k=0}^{s} h_k \int_0^1 \lambda^k f_j^*(\lambda) \, d\lambda =
\sum_{k=0}^{s} h_k \, m^k_j =
\h^\top \m_j \:, \]
and so, as $C$ is a constant,~\eqref{eq:L2} reduces to 
\[ \min_{\h \in \R^{s+1}} \h^\top \H_s \h - 2 \h^\top \m_j \: , \quad j=1,2 \:, \]
from which~\eqref{eq:psj} follows.
\end{proof}
\bibliographystyle{plain}

\end{document}